\documentclass[runningheads]{llncs}

\usepackage{amsmath, amssymb, enumerate ,mathtools,mathrsfs,xfrac,enumitem,etoolbox}
\usepackage{amsfonts}
\usepackage{enumerate}
\usepackage[ugly]{nicefrac}

\usepackage{fontawesome}

\newenvironment{proofof}[1]{\par\noindent{\itshape Proof of #1\ }}{\qed}

  \usepackage{hyperref,url}
  \hypersetup{  
    bookmarksnumbered
  }
  \hypersetup{  
    breaklinks=false,
    pdfborderstyle={/S/U/W 0},  
    citebordercolor=.235 .702 .443,
    urlbordercolor=.255 .412 .882,
    linkbordercolor=.804 .149 .19,
  }

\newcommand{\reals}{\mathbb{R}}

\newcommand{\nats}{\mathbb{N}}
\newcommand{\natz}{\mathbb{N}_{0}}

\newcommand{\indica}[1]{\mathbb{I}_{#1}}

\newcommand{\prev}[1]{\mathrm{E}_{#1}}

\newcommand{\upprev}[1]{\overline{\mathrm{E}}_{#1}}
\newcommand{\eiupprev}[1]{\overline{\mathrm{E}}^{\, \mathrm{ei}}_{#1}}
\newcommand{\ciupprev}[1]{\overline{\mathrm{E}}^{\, \mathrm{ci}}_{#1}}
\newcommand{\riupprev}[1]{\overline{\mathrm{E}}^{\, \mathrm{ri}}_{#1}}
\newcommand{\lowprev}[1]{\underline{\mathrm{E}}_{#1}}
\newcommand{\eilowprev}[1]{\underline{\mathrm{E}}^{\, \mathrm{ei}}_{#1}}
\newcommand{\cilowprev}[1]{\underline{\mathrm{E}}^{\mathrm{ci}}_{#1}}
\newcommand{\rilowprev}[1]{\underline{\mathrm{E}}^{\mathrm{ri}}_{#1}}

\newcommand{\settrans}[1]{\mathscr{T}_{#1}}

\newcommand{\statespace}{\mathscr{X}}

\newcommand{\setofgambles}{\mathscr{L}}

\newcommand{\uptrans}{\overline{T}}

\newcommand{\avuptrans}[2]{\overline{T}{}_{\hspace*{-2pt} #1}^{#2}}

\newcommand{\eimarkov}[1]{\mathscr{P}^{\, \mathrm{ei}}_{#1}}
\newcommand{\cimarkov}[1]{\mathscr{P}^{\, \mathrm{ci}}_{#1}}
\newcommand{\rimarkov}[1]{\mathscr{P}^{\, \mathrm{ri}}_{#1}}

\newcommand{\hnorm}[1]{\left\lVert#1\right\rVert_{\mathrm{H}}}
\newcommand{\supnorm}[1]{\left\lVert#1\right\rVert_{\infty}}

\newtoggle{arxiv}
\toggletrue{arxiv}

\begin{document}

\title{Limit Behaviour of Upper and Lower Expected Time Averages in Discrete-Time Imprecise Markov Chains}
\titlerunning{Expected Time Averages in Discrete-Time Imprecise Markov Chains}

\author{\Large Natan T'Joens
\and Jasper De Bock
}

\authorrunning{N. T'Joens et al.}

\institute{FLip, Ghent University, Belgium
\email{\{natan.tjoens,jasper.debock\}@ugent.be}}

\maketitle


\begin{abstract}
We study the limit behaviour of upper and lower bounds on expected time averages in imprecise Markov chains; a generalised type of Markov chain where the local dynamics, traditionally characterised by transition probabilities, are now represented by sets of `plausible' transition probabilities.
Our main result is a necessary and sufficient condition under which these upper and lower bounds, called upper and lower expected time averages, will converge as time progresses towards infinity to limit values that do not depend on the process' initial state.
Remarkably, our condition is considerably weaker than those needed to establish similar results for so-called limit---or steady state---upper and lower expectations, which are often used to provide approximate information about the limit behaviour of time averages as well.
We show that such an approximation is sub-optimal and that it can be significantly improved by directly using upper and lower expected time averages.
\end{abstract}

\section{Introduction}
Markov chains are probabilistic models that can be used to describe the uncertain dynamics of a large variety of stochastic processes.
One of the key results within the field is the point-wise ergodic theorem.
It establishes a relation between the long-term time average of a real-valued function and its limit expectation, which is guaranteed to exist if the Markov chain is ergodic.\footnote{
The term ergodicity has various meanings; sometimes it refers to properties of an invariant measure, sometimes it refers to properties such as irreducibility (with or without aperiodicity), regularity, ... 
Our usage of the term follows conventions introduced in earlier work \cite{DECOOMAN201618,Hermans:2012ie} on imprecise Markov chains; see Sections~\ref{section: precise Markov chains} and~\ref{section: trans operators and ergodicity}.}\label{footnote}   
For this reason, limit expectations and limit distributions have become central objects of interest.
Of course, if one is interested in the long-term behaviour of time averages, one could also study the expected values of these averages directly.
This is not often done though, because the limit of these expected time averages coincides with the aforementioned limit expectations, which can straightforwardly be obtained by solving a linear eigenproblem \cite{kemeny1960finite}.

We here consider a generalisation of Markov chains, called imprecise Markov chains \cite{deCooman:2009jz,HermansITIP,DECOOMAN201618}, for which the considerations above are not necessarily true.
Imprecise Markov chains are sets of traditional (``precise'') probabilistic models, where the Markov property (history independence) and time-homogeneity apply to the collection of precise models as a whole, but not necessarily to the individual models themselves.
Imprecise Markov chains therefore allow one to incorporate model uncertainty about the numerical values of the transition probabilities that make up a Markov chain, but also, and more importantly, about structural assumptions such as time-homogeneity and the Markov property.
For such an imprecise Markov chain, one is then typically interested in obtaining tight upper and lower bounds on inferences for the individual constituting models.
The operators that represent these upper and lower bounds are respectively called upper and lower expectations.

Just like traditional Markov chains can have a limit expectation, an imprecise Markov chain can have limit upper and lower expectations.
There are necessary and sufficient conditions for their existence \cite{Hermans:2012ie} as well as an imprecise variant of the point-wise ergodic theorem \cite{DECOOMAN201618}.
An important difference with traditional Markov chains however, is that upper and lower bounds on expectations of time averages---we will call these upper and lower expected time averages---may not converge to limit upper and lower expectations.
Nevertheless, because they give conservative bounds \cite[Lemma~57]{8535240}, and because they are fairly easy to compute, limit upper and lower expectations are often used as descriptors of the long-term behaviour of imprecise Markov chains, even if one is actually interested in time averages.
This comes at a cost though: as we illustrate in Section~\ref{section: trans operators and ergodicity}, both inferences can differ greatly, with limit expectations providing far too conservative bounds.

Unfortunately, apart from some experiments in \cite{8535240}, little is known about the long-term behaviour of upper and lower expected time averages in imprecise Markov chains.
The aim of this paper is to remedy this situation.
Our main result is an accessibility condition that is necessary and sufficient for upper and lower expected time averages to converge to a limit value that does not depend on the process' initial state; see Section~\ref{section: Result}.
Remarkably, this condition is considerably weaker than the ones required for limit lower and upper expectations to exist.

Technical proofs are relegated to the appendix \iftoggle{arxiv}{at the end of the paper}{of an extended online version~\cite{arXivIPMU2020}}. This is particularly true for the results in Section~\ref{section: Result}, where the main text provides an informal argument that aims to provide intuition. 

\section{Markov Chains}\label{section: precise Markov chains}
  
We consider an infinite sequence $X_{0} X_{1} X_{2} \cdots$ of uncertain states, where each state $X_k$ at time $k \in \natz \coloneqq \nats \cup \{0\}$ takes values in some finite set $\statespace{}$, called the \emph{state space}.
Such a sequence $X_{0} X_{1} X_{2} \cdots$ will be called a \emph{(discrete-time) stochastic process}.
For any $k,\ell \in \natz{}$ such that $k \leq \ell$, we use $X_{k:\ell}$ to denote the finite subsequence $X_{k} \cdots X_{\ell}$ of states that takes values in $\statespace{}^{\ell-k+1}$.
Moreover, for any $k,\ell \in \natz{}$ such that $k \leq \ell$ and any $x_{k:\ell} \in \statespace{}^{\ell-k+1}$, we use $X_{k:\ell} = x_{k:\ell}$ to denote the event that $X_{k} = x_k \cdots X_{\ell} = x_\ell$.
The uncertain dynamics of a stochastic process are then typically described by probabilities of the form $\mathrm{P}(X_{k+1} = x_{k+1} \vert X_{0:k} = x_{0:k})$, for any $k \in \natz{}$ and any $x_{0:k+1} \in \statespace{}^{k+2}$.
They represent beliefs about which state the process will be in at time $k+1$ given that we know that it was in the states $x_{0} \cdots x_{k}$ at time instances $0$ through $k$.
Additionally, our beliefs about the value of the initial state $X_0$ can be represented by probabilities $\mathrm{P}(X_0=x_0)$ for all $x_{0} \in \statespace{}$.
The local probability assessments $\mathrm{P}(X_{k+1} = x_{k+1} \vert X_{0:k} = x_{0:k})$ and $\mathrm{P}(X_0=x_0)$ can now be combined to construct a global probability model $\mathrm{P}$ that describes the dynamics of the process on a more general level.
This can be done in various ways; one of the most common ones being a measure-theoretic approach where countable additivity plays a central role.
For our purposes however, we will only require finite additivity.
Regardless, once you have such a global probability model $\mathrm{P}$, it can then be used to define expectations and make inferences about the uncertain behaviour of the process.

For any set $A$, let us write $\setofgambles{}(A)$ to denote the set of all real-valued functions on $A$.
Throughout, for any $a \in A$, we use $\indica{a}$ to denote the \emph{indicator} of $a$: the function in $\setofgambles{}(A)$ that takes the value $1$ in $a$ and $0$ otherwise.
We will only be concerned with (upper and lower) expectations of \emph{finitary functions}:
functions that depend on the state of the process at a finite number of time instances.
So if $f$ is finitary, we can write $f = g(X_{0:k})$ for some $k \in \natz{}$ and some $g \in \setofgambles{}(\statespace{}^{k+1})$.
Note that finitary functions are bounded; this follows from their real-valuedness and the fact that $\statespace{}$ is finite.
The expectation of a finitary function $f(X_{0:k})$ conditional on some event $X_{0:\ell} = x_{0:\ell}$ simply reduces to a finite weighted sum:
\begin{align*}
\prev{\mathrm{P}}(f(X_{0:k}) \vert X_{0:\ell} = x_{0:\ell})
= 
\sum_{x_{\ell+1:k} \in \statespace{}^{k-\ell}} f(x_{0:k})\prod_{i=\ell}^{k-1} \mathrm{P}(X_{i+1} = x_{i+1} \vert X_{0:i} = x_{0:i}).
\end{align*} 
A particularly interesting case arises when studying stochastic processes that are described by a probability model $\mathrm{P}$ that satisfies
\begin{align*}
\mathrm{P}(X_{k+1} = y \, \vert \,  X_{0:k} = x_{0:k}) = \mathrm{P}(X_{k+1} = y \, \vert \,  X_{k} =  x_{k}),
\end{align*}
for all $k \in \natz{}$, all $y \in \statespace{}$ and all $x_{0:k} \in \statespace{}^{k+1}$.
This property, known as the \emph{Markov property}, states that given the present state of the process the future behaviour of the process does not depend on its  history. 
A process of this type is called a \emph{Markov chain}.
We moreover call it \emph{(time) homogeneous} if additionally
$\mathrm{P}(X_{k+1} = y \, \vert \, X_{k} = x) = \mathrm{P}(X_{1} = y \, \vert \, X_{0} = x)$,
for all $k \in \natz{}$ and all $x,y \in \statespace{}$.
Hence, together with the assessments $\mathrm{P}(X_0 = x_0)$, the dynamics of a homogeneous Markov chain are fully characterised by the probabilities $\mathrm{P}(X_{1} = y \, \vert \, X_{0} = x)$.
These probabilities are typically gathered in a \emph{transition matrix} $T$; a row-stochastic $\vert \statespace{} \vert  \times \vert \statespace{} \vert$ matrix $T$ that is defined by $T(x,y) \coloneqq \mathrm{P}(X_{1} = y \, \vert \, X_{0} = x)$ for all $x,y \in \statespace{}$.
This matrix representation~$T$ is particularly convenient because it can be regarded as a linear operator from $\setofgambles{}(\statespace{})$ to $\setofgambles{}(\statespace{})$, defined for any $k \in \natz{}$, any $f \in \setofgambles{}(\statespace{})$ and any $x \in \statespace{}$ by 
\begin{align*}
Tf (x) \coloneqq 
\sum_{y \in \statespace{}} f(y) \mathrm{P}(X_{k+1} = y \, \vert X_{k} = x) 
= \prev{\mathrm{P}}(f(X_{k+1}) \, \vert \, X_{k} = x ).
\end{align*}
More generally, we have that $\prev{\mathrm{P}}(f(X_{k+\ell}) \, \vert \, X_{k} = x ) = T^{\ell} f(x)$ for all $k \in \natz{}$, all $\ell \in \natz{}$ and all $x \in \statespace{}$.
Then, under some well-known accessibility conditions \cite[Proposition~3]{Hermans:2012ie}, the expectation $T^{\ell} f(x)$ converges for increasing $\ell$ towards a constant $\prev{\infty}(f)$ independently of the initial state $x$.
If this is the case for all $f \in \setofgambles{}(\statespace{})$, the homogeneous Markov chain will have a steady-state distribution, represented by the limit expectation $\prev{\infty}$, and we call the Markov chain \emph{ergodic}. 
The expectation $\prev{\infty}$ is in particular also useful if we are interested in the limit behaviour of expected time averages.
Indeed, let $\overline{f}_{k}(X_{\ell:\ell + k}) \coloneqq \sfrac{1}{(k+1)} \sum_{i = \ell}^{\ell + k} f(X_i)$ be the time average of some function $f \in \setofgambles{}(\statespace{})$ evaluated at the time instances $\ell$ through $k+\ell$.
Then, according to \cite[Theorem~38]{8535240}, the limit of the expected average $\lim_{k \to +\infty} \prev{\mathrm{P}}(\overline{f}_{k}(X_{0:k}))$ coincides with the limit expectation $\prev{\infty}(f)$.
One of the aims of this paper is to explore to which extent this remains true for imprecise Markov chains.

\section{Imprecise Markov Chains}

If the basic probabilities $\mathrm{P}(X_{k+1} \vert X_{0:k} = x_{0:k})$ that describe a stochastic process are imprecise, in the sense that we only have partial information about them, then we can still model the process' dynamics by considering a set $\settrans{x_{0:k}}$ of such probabilities, for all $k \in \natz{}$ and all $x_{0:k} \in \statespace{}^{k+1}$.
This set $\settrans{x_{0:k}}$ is then interpreted as the set of all probability mass functions $\mathrm{P}(X_{k+1} \vert X_{0:k} = x_{0:k})$ that we deem ``plausible''.
We here consider the special case where the sets $\settrans{x_{0:k}}$ satisfy a Markov property, meaning that $\settrans{x_{0:k}} = \settrans{x_{k}}$ for all $k \in \natz{}$ and all $x_{0:k} \in \statespace{}^{k+1}$.
Similar to the precise case, the sets $\settrans{x}$, for all $x \in \statespace{}$, can be gathered into a single object: the set $\settrans{}$ \/ of all row stochastic $\vert \statespace{} \vert \times \vert \statespace{} \vert$ matrices $T$ such that, for all $x \in \statespace{}$, the probability mass function $T(x, \cdot )$ is an element of $\settrans{x}$. 
A set $\settrans{}$ \/ of transition matrices defined in this way is called \emph{separately specified} \cite{HermansITIP}.
For any such set $\settrans{}$, the corresponding \emph{imprecise Markov chain under epistemic irrelevance} $\eimarkov{\settrans{}}$ \cite{deCooman:2010gd} is the set of all (precise) probability models $\mathrm{P}$ such that $\smash{\mathrm{P}(X_{k+1} \vert X_{0:k} = x_{0:k}) \in \settrans{x_k}}$ for all $k \in \natz{}$ and all $\smash{x_{0:k} \in \statespace{}^{k+1}}$.
The values of the probabilities $\mathrm{P}(X_0 = x_0)$ will be of no importance to us, because we will focus solely on (upper and lower) expectations conditional on the value of the initial state $X_0$.

Clearly, an imprecise Markov chain $\eimarkov{\settrans{}}$ also contains non-homogeneous, and even non-Markovian processes.
So the Markov property does in this case not apply to the individual probability assessments, but rather to the sets $\settrans{x_{0:k}}$.
The model $\eimarkov{\settrans{}}$ is therefore a generalisation of a traditional Markov chain where we allow for model uncertainty about, on the one hand, the mass functions $\mathrm{P}(X_{k+1} \vert X_{0:k} = x_{0:k})$ and, on the other hand, about structural assumptions such as the Markov and time-homogeneity property.
However, there are also types of imprecise Markov chains that do impose some of these properties.
For a given set $\settrans{}$, the \emph{imprecise Markov chain under complete independence} $\smash{\cimarkov{\settrans{}}}$ is the subset of $\smash{\eimarkov{\settrans{}}}$ that contains all, possibly non-homogeneous, Markov chains in $\smash{\eimarkov{\settrans{}}}$~\cite{8535240}.
The \emph{imprecise Markov chain under repetition independence} $\smash{\rimarkov{\settrans{}}}$ is the subset of $\smash{\eimarkov{\settrans{}}}$ containing all homogeneous Markov chains \cite{8535240}.
Henceforth, we let $\settrans{}$ \/ be some fixed, arbitrary set of transition matrices that is separately specified. 


Now, for any probability model $\mathrm{P}$ in the imprecise Markov chain $\eimarkov{\settrans{}}$, we can again consider the corresponding expectation operator $\prev{\mathrm{P}}$.
The \emph{upper} and \emph{lower expectation} are then respectively defined as the tightest upper and lower bound on this expectation:
\begin{align*}
\eiupprev{\settrans{}} \, (f \/ \vert A) \coloneqq \sup_{\mathrm{P} \in \eimarkov{\settrans{}}} \prev{\mathrm{P}}(f \/ \vert  A) \quad 
\text{ and } \quad \eilowprev{\settrans{}} \,(f  \vert  A) \coloneqq \inf_{\mathrm{P} \in \eimarkov{\settrans{}}} \prev{\mathrm{P}}(f  \vert  A),
\end{align*}
for any finitary function $f$ and any event $A$ of the form $X_{0:k} = x_{0:k}$. 
The operators $\smash{\eiupprev{\settrans{}}}$ and $\smash{\eilowprev{\settrans{}}}$ are related by conjugacy, meaning that $\smash{\eilowprev{\settrans{} \,}(\cdot \vert \cdot) = - \eiupprev{\settrans{} \,}(- \cdot \vert \cdot)}$, which allows us to focus on only one of them; upper expectations in our case.
The lower expectation $\smash{\eilowprev{\settrans{} \,}(f \/ \vert A)}$ of a finitary function $f$ can then simply be obtained by considering the upper expectation $\smash{- \eiupprev{\settrans{} \,}(- f \/ \vert A)}$.

In a similar way, we can define the upper expectations $\smash{\ciupprev{\settrans{}}}$ and $\smash{\riupprev{\settrans{}}}$ and the lower expectations $\smash{\cilowprev{\settrans{}}}$ and $\smash{\rilowprev{\settrans{}}}$ as the tightest upper and lower bounds on the expectations corresponding to the models in $\cimarkov{\settrans{}}$ and $\rimarkov{\settrans{}}$, respectively.
Since $\rimarkov{\settrans{}} \subseteq \cimarkov{\settrans{}} \subseteq \eimarkov{\settrans{}}$, we have that $\smash{\riupprev{\settrans{}} \, (f \/ \vert A) 
\leq \ciupprev{\settrans{}} \, (f \/ \vert A)
\leq \eiupprev{\settrans{}} \, (f \/ \vert A)}$
for any finitary function $f$ and any event $A$ of the form $X_{0:k} = x_{0:k}$.

As we have mentioned before, imprecise Markov chains generalise traditional Markov chains by incorporating different types of model uncertainty. 
The corresponding upper (and lower) expectations then allow us to make inferences that are robust with respect to this model uncertainty.
For a more detailed discussion on the motivation for and interpretation behind these and other types of so-called imprecise probability models, we refer to \cite{troffaes2014,Walley:1991vk,Augustin:2014di}. 

Within the context of imprecise Markov chains, we will be specifically concerned with two types of inferences: the upper and lower expectation of a function at a single time instant, and the upper and lower expectation of the time average of a function.
For imprecise Markov chains under epistemic irrelevance and under complete independence, both of these inferences coincide \cite[Theorem~51 \& Theorem~52]{8535240}.
For any $f \in \setofgambles{}(\statespace{})$ and any $x \in \statespace{}$, we will denote them by
\vspace*{-4pt}
\begin{align*}
\upprev{k}(f \vert x)
= \, &\eiupprev{\settrans{}}(f(X_{k}) \vert X_0 = x) 
= \ciupprev{\settrans{}}(f(X_{k}) \vert X_0 = x) \\
\text{ and } \  
\upprev{\mathrm{av},k}(f \vert x)
= \, &\eiupprev{\settrans{}}(\overline{f}_{k}(X_{0:k}) \vert X_0 = x) 
= \ciupprev{\settrans{}}(\overline{f}_{k}(X_{0:k}) \vert X_0 = x),
\end{align*}
respectively, where the dependency on $\settrans{}$ \/ is implicit.
The corresponding lower expectations can be obtained through conjugacy: 
$\lowprev{k}(f \vert x) = - \upprev{k}( -f \vert x)$ and $\lowprev{\mathrm{av},k}(f \vert x) = - \upprev{\mathrm{av},k}( -f \vert x)$ for all $f \in \setofgambles{}(\statespace{})$ and all $x \in \statespace{}$.
In the remainder, we will omit imprecise Markov chains under repetition independence from the discussion.
Generally speaking, this type of imprecise Markov chain is less studied within the field of imprecise probability because of its limited capacity to incorporate model uncertainty.
Indeed, it is simply a set of time-homogeneous precise Markov chains and therefore only allows for model uncertainty about the numerical values of the transition probabilities.
Moreover, as far as we know, a characterisation for the ergodicity of such Markov chains---a central topic in this paper---is currently lacking.
We therefore believe that this subject demands a separate discussion, which we defer to future work.

\section{Transition Operators, Ergodicity and Weak Ergodicity}\label{section: trans operators and ergodicity}
 
Inferences of the form $\upprev{k}(f \vert x)$ were among the first ones to be thoroughly studied in imprecise Markov chains.
Their study was fundamentally based on the observation that $\upprev{k}(f \vert x)$ can be elegantly rewritten as the $k$-th iteration of the map $\uptrans{} \colon \setofgambles{}(\statespace{}) \to \setofgambles{}(\statespace{})$ defined by 
\begin{align*}
\uptrans{} h (x) 
\coloneqq \sup_{T \in \settrans{}} \, T h (x)
= \sup_{T(x , \cdot) \in \settrans{x}} \, \sum_{y \in \statespace{}} T(x,y) h(y),
\end{align*}
for all $x \in \statespace{}$ and all $h \in \setofgambles{}(\statespace{})$.
Concretely, $\upprev{k}(f \vert x) = [\uptrans{}^k f](x)$ for all $x \in \statespace{}$ and all $k \in \natz{}$ \cite[Theorem~3.1]{deCooman:2009jz}.
The map $\uptrans{}$ therefore plays a similar role as the transition matrix $T$ in traditional Markov chains, which is why it is called the \emph{upper transition operator} corresponding to the set $\settrans{}$.

In an analogous way, inferences of the form $\upprev{\mathrm{av},k}(f \vert x)$ can be obtained as the $k$-th iteration of the map $\smash{\avuptrans{f}{} \colon \setofgambles{}(\statespace{}) \to \setofgambles{}(\statespace{})}$ defined by $\smash{\avuptrans{f}{} h \coloneqq f + \uptrans{} h}$ for all $h \in \setofgambles{}(\statespace{})$.
In particular, if we let $\tilde{m}_{f,0} \coloneqq f = \avuptrans{f}{} (0)$ and 
\begin{align}\label{Eq: recursive expression}
\tilde{m}_{f,k} \coloneqq f + \uptrans{} \tilde{m}_{f,k-1} = \avuptrans{f}{} \tilde{m}_{f,k-1} \text{ for all } k \in \nats{},
\end{align} 
then it follows from \cite[Lemma 41]{8535240} that $\upprev{\mathrm{av},k}(f \vert x) = \tfrac{1}{k+1} \tilde{m}_{f,k}(x)$ for all $x \in \statespace{}$ and all $k \in \natz{}$.
Applying Equation~\eqref{Eq: recursive expression} repeatedly, we find that for all $x \in \statespace{}$:
\begin{align}\label{Eq: recursive expression 2}
\upprev{\mathrm{av},k}(f \vert x) = 
\tfrac{1}{k+1}\tilde{m}_{f,k}(x)
=\tfrac{1}{k+1} [\avuptrans{f}{k} \tilde{m}_{f,0}](x) = \tfrac{1}{k+1} [\avuptrans{f}{k+1}(0)](x).
\end{align}
The same formula can also be obtained as a special case of the results in \cite{10.1007/978-3-030-29765-7_38}.

These expressions for $\upprev{k}(f \vert x)$ and $\upprev{\mathrm{av},k}(f \vert x)$ in terms of the respective operators $\uptrans{}$ and $\avuptrans{f}{}$ are particularly useful when we aim to characterise the limit behaviour of these inferences.
As will be elaborated on in the next section, there are conditions on $\uptrans{}$ that are necessary and sufficient for $\upprev{k}(f \vert x)$ to converge to a limit value that does not depend on the process' initial state $x \in \statespace{}$.
If this is the case for all $f \in \setofgambles{}(\statespace{})$, the imprecise Markov chain is called \emph{ergodic} and we then denote the constant limit value by $\upprev{\infty}(f) \coloneqq \lim_{k \to +\infty} \upprev{k}(f \vert x)$.
Similarly, we call an imprecise Markov chain \emph{weakly ergodic} if, for all $f \in \setofgambles{}(\statespace{})$, $\lim_{k \to +\infty} \upprev{\mathrm{av},k}(f \vert x)$ exists and does not depend on the initial state $x$.
For a weakly ergodic imprecise Markov chain, we denote the common limit value by $\upprev{\mathrm{av},\infty}(f) \coloneqq \lim_{k \to +\infty} \upprev{\mathrm{av},k}(f \vert x)$.
In contrast with standard ergodicity, weak ergodicity and, more generally, the limit behaviour of $\upprev{\mathrm{av},k}(f \vert x)$, is almost entirely unexplored.
The aim of this paper is to remedy this situation.
The main contribution will be a necessary and sufficient condition for an imprecise Markov chain to be weakly ergodic.
As we will see, this condition is weaker than those needed for standard ergodicity, hence our choice of terminology.
The following example shows that this difference already becomes apparent in the precise case.
\begin{example}\label{example 1}
Let $\statespace{} = \{a,b\}$, consider any function $\smash{f = \big[\begin{smallmatrix}
f_a \\ f_b
\end{smallmatrix}\big] \in \setofgambles{}(\statespace{})}$ and assume that $\settrans{}$ consists of a single matrix 
$T = 
\big[\begin{smallmatrix}
0 & 1 \\ 1 & 0
\end{smallmatrix}\big]$.
Clearly, $\uptrans{}$ is not ergodic because $\smash{\uptrans{}^{(2\ell + 1)} f
= T^{(2\ell + 1)} f
= \big[\begin{smallmatrix}
0 & 1 \\ 1 & 0
\end{smallmatrix}\big]f
= \big[\begin{smallmatrix}
f_b \\ f_a
\end{smallmatrix}\big]}
$ and 
$\smash{\uptrans{}^{(2\ell)} f 
= \big[\begin{smallmatrix}
1 & 0 \\ 0 & 1
\end{smallmatrix}\big]f
= \big[\begin{smallmatrix}
f_a \\ f_b
\end{smallmatrix}\big]}$
for all $\ell \in \natz{}$.
$\uptrans{}$ is weakly ergodic though, because
\begin{align*}
\avuptrans{f}{(2\ell)}(0) = \ell \big[\begin{smallmatrix}
f_a + f_b \\ f_a + f_b
\end{smallmatrix}\big]
 \, \text{ and } \, \avuptrans{f}{(2\ell +1)}(0) = f + \uptrans{} \, \avuptrans{f}{(2\ell)}(0) = f + \ell \big[\begin{smallmatrix}
f_a + f_b \\ f_a + f_b
\end{smallmatrix}\big],
\end{align*} 
for all $\ell \in \natz{}$, which implies that $\upprev{\mathrm{av},\infty}(f) \coloneqq \lim_{k \to +\infty} \avuptrans{f}{k}(0)/k = ( f_a + f_b )/2$ exists.
\hfill $\Diamond$
\end{example}

Notably, even if an imprecise Markov chain is ergodic (and hence also weakly ergodic) and therefore both $\upprev{\infty}(f)$ and $\upprev{\mathrm{av},\infty}(f)$ exist, these inferences will not necessarily coincide.
This was first observed in an experimental setting \cite[Section~7.6]{8535240}, but the differences that were observed there were marginal.
The following example shows that these differences can in fact be very substantial.

\begin{example}\label{example 2}
Let $\statespace{} = \{a,b\}$, let $\settrans{a}$ be the set of all probability mass functions on $\statespace{}$ and let $\settrans{b} \coloneqq \{ p \}$ for the probability mass function $p = (p_a , p_b) = (1,0)$ that puts all mass in $a$.
Then, for any $f = \big[\begin{smallmatrix}
f_a \\ f_b
\end{smallmatrix}\big] \in \setofgambles{}(\statespace{})$, we have that
\begin{align*}
\uptrans{}f (x) = 
\begin{aligned}
\begin{cases}
\max f &\text{ if } x=a; \\
f_a &\text{ if } x=b,
\end{cases}
\end{aligned}
\quad \text{ and } \quad
\uptrans{}^{ \, 2} f (x) = 
\begin{aligned}
\begin{cases}
\max \uptrans{} f = \max f &\text{ if } x=a; \\
\uptrans{}f (a) = \max f &\text{ if } x=b.
\end{cases}
\end{aligned}
\end{align*}
It follows that $\uptrans{}^k f = \max f$ for all $k \geq 2$, so the limit upper expectation $\upprev{\infty}(f)$ exists and is equal to $\max f$ for all $f \in \setofgambles{}(\statespace{})$.
In particular, we have that $\upprev{\infty}(\indica{b}) = 1$.  
On the other hand, we find that 
$\smash{\uptrans{}_{\indica{b}}^{(2\ell)}(0) = \ell}$ and $\smash{\uptrans{}_{\indica{b}}^{(2\ell + 1)}(0)} = \smash{\indica{b} + \uptrans{} \, \uptrans{}_{\indica{b}}^{(2\ell)}(0)} = \smash{\big[\begin{smallmatrix}
\ell \\ \ell+1
\end{smallmatrix}\big]}$ for all $\ell \in \natz{}$.
This implies that the upper expectation $\smash{\upprev{\mathrm{av},\infty}(\indica{b})} \coloneqq \smash{\lim_{k \to +\infty} \uptrans{}_{\indica{b}}^k(0)/k}$ exists and is equal to $1/2$.
This value differs significantly from the limit upper expectation $\upprev{\infty}(\indica{b}) = 1$.

In fact, this result could have been expected simply by taking a closer look at the dynamics that correspond to $\settrans{}$. \/
Indeed, it follows directly from $\settrans{}$ \/ that, if the system is in state $b$ at some instant, then it will surely be in $a$ at the next time instant.
Hence, the system can only reside in state $b$ for maximally half of the time, resulting in an upper expected average that converges to $1/2$.
These underlying dynamics have little effect on the limit upper expectation $\upprev{\infty}(\indica{b})$ though, because it is only concerned with the upper expectation of $\indica{b}$ evaluated at a single time instant.
\hfill $\Diamond$
\end{example}

Although we have used sets $\settrans{}$ of transition matrices to define imprecise Markov chains, it should at this point be clear that, if we are interested in the inferences $\upprev{k}(f \vert x)$ and $\upprev{\mathrm{av},k}(f \vert x)$ and their limit values, then it suffices to specify $\uptrans{}$.
In fact, we will henceforth forget about $\settrans{}$ \/ and will assume that $\uptrans{}$ is a \emph{coherent} upper transition operator on $\setofgambles{}(\statespace{})$, meaning that it is an operator from $\setofgambles{}(\statespace{})$ to $\setofgambles{}(\statespace{})$ that satisfies
\begin{enumerate}[leftmargin=*,ref={\upshape{}C\arabic*},label={\upshape{}C\arabic*}.,itemsep=3pt, series=transcoherence]
\item\label{transcoherence: bounds} $\min h  \leq \uptrans{} h \leq \max h$ \hfill [boundedness];
\item\label{transcoherence: subadditivity} $\uptrans{} (h+g) \leq \uptrans{}h + \uptrans{}g$ \hfill [sub-additivity];
\item\label{transcoherence: homogeneity} $\uptrans{}(\lambda h) = \lambda \uptrans{}h$ \hfill [non-negative homogeneity],
\end{enumerate}
for all $h,g \in \setofgambles{}(\statespace)$ and all real $\lambda \geq 0$ \cite{Walley:1991vk,Williams:2007eu,troffaes2014}, and we will regard $\upprev{k}(f \vert x)$ and $\upprev{\mathrm{av},k}(f \vert x)$ as objects that correspond to $\uptrans{}$. 
Our results and proofs will never rely on the fact that $\uptrans{}$ is derived from a set $\settrans{}$ of transition matrices, but will only make use of \ref{transcoherence: bounds}-\ref{transcoherence: homogeneity} and the following \iftoggle{arxiv}{three}{two} properties that are implied by them \cite[Section~2.6.1]{Walley:1991vk}:
\begin{enumerate}[leftmargin=*,ref={\upshape{}C\arabic*},label={\upshape{}C\arabic*}.,itemsep=3pt, resume=transcoherence]
\item\label{transcoherence: constant addivity} $\uptrans{}(\mu + h) = \mu + \uptrans{}h$ \hfill [constant additivity];
\item\label{transcoherence: monotonicity} if $h \leq g$ then $\uptrans{} h \leq \uptrans{} g$ \hfill [monotonicity]\iftoggle{arxiv}{;}{,}
\iftoggle{arxiv}{
\item\label{transcoherence: mixed additivity} $\uptrans{}h - \uptrans{}g \leq \uptrans{} (h-g)$ \hfill [mixed sub-additivity],}{}
\end{enumerate}
for all $h,g \in \setofgambles{}(\statespace)$ and all real $\mu$.
This can be done without loss of generality because an upper transition operator $\smash{\uptrans{}}$ that is defined as an upper envelope of a set $\settrans{}$ of transition matrices---as we did in Section~\ref{section: trans operators and ergodicity}---is always coherent \cite[Theorem~2.6.3]{Walley:1991vk}.
Since properties such as ergodicity and weak ergodicity can be completely characterised in terms of $\uptrans{}$, we will henceforth simply say that $\uptrans{}$ itself is (weakly) ergodic, instead of saying that the corresponding imprecise Markov chain is.

\section{Accessibility Relations and Topical Maps}\label{Sect: accessibility and topical maps}

To characterise ergodicity and weak ergodicity, we will make use of some well-known graph-theoretic concepts, suitably adapted to the imprecise Markov chain setting; we recall the following from \cite{deCooman:2009jz} and \cite{Hermans:2012ie}.
The \emph{upper accessibility graph} $\mathscr{G}(\uptrans{})$ corresponding to $\uptrans{}$ is defined as the graph with vertices $x_1 \cdots x_n \in \statespace{}$, where $n \coloneqq \vert \statespace{} \vert$, with an edge from $x_i$ to $x_j$ if $\uptrans{}\indica{x_j}(x_i) > 0$.
For any two vertices $x_i$ and $x_j$, we say that $x_j$ is \emph{accessible} from $x_i$, denoted by $x_i \to x_j$, if $x_i = x_j$ or if there is a directed path from $x_i$ to $x_j$, which means that there is a sequence $x_i = x'_0, x'_1 , \cdots, x'_m = x_j$ of vertices, with $m \in \nats{}$, such that there is an edge from $x'_{\ell-1}$ to $x'_{\ell}$ for all $\ell \in \{1,\cdots,m\}$.
We say that two vertices $x_i$ and $x_j$ \emph{communicate} and write $x_i \leftrightarrow x_j$ if both $x_i \to x_j$ and $x_j \to x_i$.
The relation $\leftrightarrow$ is an equivalence relation (reflexive, symmetric and transitive) and the equivalence classes are called \emph{communication classes}.
We call the graph $\mathscr{G}(\uptrans{})$ \emph{strongly connected} if any two vertices $x_i$ and $x_j$ in $\mathscr{G}(\uptrans{})$ communicate, or equivalently, if $\statespace{}$ itself is a communication class. 
Furthermore, we say that $\uptrans{}$ (or $\mathscr{G}(\uptrans{})$) has a \emph{top class} $\mathcal{R}$ if
\begin{align*}
\mathcal{R} \coloneqq \{ x \in \statespace{} \colon y \to x \text{ for all } y \in \statespace{} \} \not= \emptyset.
\end{align*}
So, if $\uptrans{}$ has a top class $\mathcal{R}$, then $\mathcal{R}$ is accessible from any vertex in the graph $\mathscr{G}(\uptrans{})$.
As a fairly immediate consequence, it follows that $\mathcal{R}$ is a communication class that is \emph{maximal} or \emph{undominated}, meaning that $x \not\to y$ for all $x \in \mathcal{R}$ and all $y \in \mathcal{R}^c$.
In fact, it is the only such maximal communication class.

Having a top class is necessary for $\uptrans{}$ to be ergodic, but it is not sufficient.
Sufficiency additionally requires that the top class $\mathcal{R}$ satisfies \cite[Proposition~3]{Hermans:2012ie}: 
\begin{enumerate}[itemsep=3pt,align=parleft,labelwidth=1.4em,leftmargin =\dimexpr\labelwidth+\labelsep\relax]
\item[E1.] $(\forall x \in \mathcal{R})(\exists k^\ast \in \nats{})(\forall k \geq k^\ast) \ \min \uptrans{}^k \indica{x} > 0$ \hfill [Regularity];
\item[E2.] $(\forall x \in \mathcal{R}^c)(\exists k \in \nats{}) \  \uptrans{}^k \indica{\mathcal{R}^c}(x) < 1$ \hfill [Absorbing].
\end{enumerate} 
We will say that $\uptrans{}$ is \emph{top class regular} (TCR) if it has a top class that is regular, and analogously for \emph{top class absorbing} (TCA).
Top class regularity represents aperiodic behaviour: it demands that there is some time instant $k^\ast \in \nats{}$ such that all of the elements in the top class $\mathcal{R}$ are accessible from each other in $k$ steps, for any $k \geq k^\ast$.
In the case of traditional Markov chains, top class regularity suffices as a necessary and sufficient condition for ergodicity \cite{kemeny1960finite,deCooman:2009jz}.
However, in the imprecise case, we need the additional condition of being top class absorbing, which ensures that the top class will eventually be reached.  
It requires that, if the process starts from any state $x \in \mathcal{R}^c$, the lower probability that it will ever transition to $\mathcal{R}$ is strictly positive.
We refer to \cite{deCooman:2009jz} for more details.
From a practical point of view, an important feature of both of these accessibility conditions is that they can be easily checked in practice~\cite{Hermans:2012ie}.

The characterisation of ergodicity using (TCR) and (TCA) was strongly inspired by the observation that upper transition operators are part of a specific collection of order-preserving maps, called \emph{topical maps}.
These are maps $F \colon \reals{}^n \to \reals{}^n$ that satisfy 
\begin{enumerate}[leftmargin=*,ref={\upshape{}T\arabic*},label={\upshape{}T\arabic*}.,itemsep=3pt, series=topical]
\item\label{topical: constant addivity} $F(\mu + h) = \mu + Fh$ \hfill [constant additivity];
\item\label{topical: monotonicity} if $h \leq g$ then $F(h) \leq F(g)$ \hfill [monotonicity],
\end{enumerate} 
for all $h,g \in \reals{}^n$ and all $\mu \in \reals{}$.
To show this, we identify $\setofgambles{}(\statespace{})$ with the finite-dimensional linear space $\reals{}^n$, with $n = \vert \statespace{} \vert$; this is clearly possible because both are isomorph.
That every coherent upper transition operator is topical now follows trivially from \ref{transcoherence: constant addivity} and \ref{transcoherence: monotonicity}.
What is perhaps less obvious, but can be derived in an equally trivial way, is that the operator $\avuptrans{f}{}$ is also topical.
This allows us to apply results for topical maps to $\smash{\avuptrans{f}{}}$ in order to find necessary and sufficient conditions for weak ergodicity.

\section{A Sufficient Condition for Weak Ergodicity}

As a first step, we aim to find sufficient conditions for the existence of $\upprev{\mathrm{av},\infty}(f)$.
To that end, recall from Section~\ref{section: trans operators and ergodicity} that if $\upprev{\mathrm{av},\infty}(f)$ exists, it is equal to the limit $\smash{\lim_{k \to +\infty} \avuptrans{f}{k} (0) / k}$.
Then, since $\avuptrans{f}{}$ is topical, the following lemma implies that it is also equal to $\smash{\lim_{k \to +\infty} \avuptrans{f}{k} h / k}$ for any $h \in \setofgambles{}(\statespace{})$.
\begin{lemma}\label{lemma: cycle time exists independently of starting point}\emph{\cite[Lemma 3.1]{GUNAWARDENA2003141}}
Consider any topical map $F \colon \reals{}^n \to \reals{}^n$. 
If the limit $\lim_{k \to +\infty} F^k h / k$ exists for some $h \in \reals{}^n$, then the limit exists for all $h \in \reals{}^n$ and they are all equal.
\end{lemma}
Hence, if $\smash{\lim_{k \to +\infty} \avuptrans{f}{k} h / k}$ converges to a constant vector $\mu$ for some $h \in \setofgambles{}(\statespace{})$, then $\upprev{\mathrm{av},\infty}(f)$ exists and is equal to $\mu$.
This condition is clearly satisfied if the map $\smash{\avuptrans{f}{}}$ has an (additive) eigenvector $h \in \setofgambles{}(\statespace{})$, meaning that $\smash{\avuptrans{f}{k} h = h + k \mu}$ for some $\mu \in \reals{}$ and all $k \in \natz{}$.
In that case, we have that $\upprev{\mathrm{av},\infty}(f) = \mu$, where $\mu$ is called the eigenvalue corresponding to $h$.

To find conditions that guarantee the existence of an eigenvector of $\avuptrans{f}{}$, we will make use of results from \cite{10.2307/3845053} and \cite{GUNAWARDENA2003141}.
There, accessibility graphs are defined in a slightly different way:
for any topical map $F \colon \reals{}^n \to \reals{}^n$, they let $\mathscr{G}'(F)$ be the graph with vertices $v_{1} , \cdots , v_{n}$ and an edge from $v_i$ to $v_j$ if $\smash{\lim_{\alpha \to +\infty} [F(\alpha \indica{v_j})](v_i) = +\infty}$.
Subsequently, for such a graph $\mathscr{G}'(F)$, the accessibility relation $\cdot \to \cdot$ and corresponding notions (e.g. `strongly connected', `top class', \dots) are defined as in Section~\ref{Sect: accessibility and topical maps}.
If we identify the vertices $v_{1} , \cdots , v_{n}$ in $\mathscr{G}'(\uptrans{})$ and $\smash{\mathscr{G}'(\avuptrans{f}{})}$ with the different states $x_{1} , \cdots , x_{n}$ in $\statespace{}$, this can in particular be done for the topical maps $\uptrans{}$ and $\avuptrans{f}{}$.
The following results show that the resulting graphs coincide with the one defined in Section~\ref{Sect: accessibility and topical maps}.

\begin{lemma}\label{lemma: edge}
For any two vertices $x$ and $y$ in $\mathscr{G}'(\uptrans{})$, there is an edge from $x$ to $y$ in $\mathscr{G}'(\uptrans{})$ if and only if there is an edge from $x$ to $y$ in $\mathscr{G}(\uptrans{})$. 
\end{lemma}
\begin{proof}
Consider any two vertices $x$ and $y$ in the graph $\mathscr{G}'(\uptrans{})$.
Then there is an edge from $x$ to $y$ if $\lim_{\alpha \to +\infty} [\uptrans{}(\alpha \indica{y})](x) = +\infty$. 
By non-negative homogeneity [\ref{transcoherence: homogeneity}], this is equivalent to the condition that $\lim_{\alpha \to +\infty} \alpha [\uptrans{}\indica{y}](x) = +\infty$.
Since moreover $0 \leq \uptrans{}\indica{y} \leq 1$ by \ref{transcoherence: bounds}, this condition reduces to $\uptrans{}\indica{y}(x) > 0$.
\qed
\end{proof}

\begin{corollary}\label{corollary: graphs are identical}
The graphs $\mathscr{G}'(\avuptrans{f}{})$, $\mathscr{G}'(\uptrans{})$ and $\mathscr{G}(\uptrans{})$ are identical.
\end{corollary}
\begin{proof}
Lemma~\ref{lemma: edge} implies that $\mathscr{G}'(\uptrans{})$ and $\mathscr{G}(\uptrans{})$ are identical.
Moreover, that $\mathscr{G}'(\avuptrans{f}{})$ is equal to $\mathscr{G}'(\uptrans{})$, follows straightforwardly from the definition of $\avuptrans{f}{}$.
\qed
\end{proof}
 
In principle, we could use this result to directly obtain the desired condition for the existence of an eigenvector from \cite[Theorem~2]{10.2307/3845053}.
However, \cite[Theorem~2]{10.2307/3845053} is given in a multiplicative framework and would need to be reformulated in an additive framework in order to be applicable to the map $\avuptrans{f}{}$; see \cite[Section~2.1]{10.2307/3845053}.
This can be achieved with a bijective transformation, but we prefer to not do so because it would require too much extra terminology and notation.
Instead, we will derive an additive variant of \cite[Theorem~2]{10.2307/3845053} directly from \cite[Theorem~9]{10.2307/3845053} and \cite[Theorem~10]{10.2307/3845053}.

The first result establishes that the existence of an eigenvector is equivalent to the fact that trajectories are bounded with respect to the Hilbert semi-norm $\hnorm{\cdot}$, defined by $\hnorm{h} \coloneqq \max h - \min h$ for all $h \in \reals{}^n$.

\begin{theorem}\label{theorem: eigenvector iff bounded}\emph{\cite[Theorem~9]{10.2307/3845053}}
Let $F \colon \reals{}^n \to \reals{}^n$ be a topical map.
Then $F$ has an eigenvector in $\reals{}^n$ if and only if $\left\{\hnorm{F^k h} \colon k \in \nats{} \right\}$ is bounded for some (and hence all) $h \in \reals{}^n$.
\end{theorem}
That the boundedness of a single trajectory indeed implies the boundedness of all trajectories follows from the non-expansiveness of a topical map with respect to the Hilbert semi-norm \cite{10.2307/3845053}.
The second result that we need uses the notion of a \emph{super-eigenspace}, defined for any topical map $F$ and any $\mu \in \reals{}$ as the set $S^\mu(F) \coloneqq \{h \in \reals{}^n \colon F h \leq h + \mu\}$.

\begin{theorem}\label{theorem: strongly connected then supereigenspace bounded}\emph{\cite[Theorem~10]{10.2307/3845053}}
Let $F \colon \reals{}^n \to \reals{}^n$ be a topical map such that the associated graph $\mathscr{G}'(F)$ is strongly connected.
Then all of the super-eigenspaces are bounded in the Hilbert semi-norm.
\end{theorem}
Together, these theorems imply that any topical map $F \colon \reals{}^n \to \reals{}^n$ for which the graph $\mathscr{G}'(F)$ is strongly connected, has an eigenvector.
The connection between both is provided by the fact that trajectories cannot leave an eigenspace.
The following result formalises this.

\begin{theorem}\label{theorem: strongly connected then eigenvector}
Let $F \colon \reals{}^n \to \reals{}^n$ be a topical map such that the associated graph $\mathscr{G}'(F)$ is strongly connected.
Then $F$ has an eigenvector in $\reals{}^n$.
\end{theorem}
\begin{proof}
Consider any $h\in\reals{}^n$ and any $\mu\in\reals$ such that $\max(Fh - h) \leq \mu$. Then $Fh\leq h+\mu$, so $h\in S^\mu(F)$.
Now notice that $F(Fh) \leq F(h +\mu) = Fh +\mu$ because of \ref{topical: constant addivity} and \ref{topical: monotonicity}, which implies that also $Fh \in S^\mu(F)$.
In the same way, we can also deduce that $F^{2} h \in S^\mu(F)$ and, by repeating this argument, that the whole trajectory corresponding to $h$ remains in $S^\mu(F)$.
This trajectory is bounded because of Theorem~\ref{theorem: strongly connected then supereigenspace bounded}, which by Theorem~\ref{theorem: eigenvector iff bounded} guarantees the existence of an eigenvector. 
\qed
\end{proof}
In particular, if $\mathscr{G}'(\avuptrans{f}{})$ is strongly connected then $\avuptrans{f}{}$ has an eigenvector, which on its turn implies the existence of $\smash{\upprev{\mathrm{av},\infty}(f)}$ as explained earlier.
If we combine this observation with Corollary~\ref{corollary: graphs are identical}, we obtain the following result.

\begin{proposition}\label{proposition: strongly connected then eigenvector}
An upper transition operator $\uptrans{}$ is weakly ergodic if the associated graph $\mathscr{G}(\uptrans{})$ is strongly connected.
\end{proposition}
\begin{proof}
Suppose that $\mathscr{G}(\uptrans{})$ is strongly connected.
Then, by Corollary~\ref{corollary: graphs are identical}, $\mathscr{G}'(\avuptrans{f}{})$ is also strongly connected.
Hence, since $\smash{\avuptrans{f}{}}$ is a topical map, Theorem~\ref{theorem: strongly connected then eigenvector} guarantees the existence of an eigenvector of $\smash{\avuptrans{f}{}}$.
As explained in the beginning of this section, this implies by Lemma~\ref{lemma: cycle time exists independently of starting point} that $\smash{\upprev{\mathrm{av},\infty}(f)}$ exists, so we indeed find that $\uptrans{}$ is weakly ergodic.
\qed 
\end{proof}

In the remainder of this paper, we will use the fact that $\uptrans{}$ is coherent---so not just topical---to strengthen this result.
In particular, we will show that the condition of being strongly connected can be replaced by a weaker one: being top class absorbing.
It will moreover turn out that this property is not only sufficient, but also necessary for weak ergodicity.

\section{Necessary and Sufficient Condition for Weak Ergodicity}\label{section: Result}

In order to gain some intuition about how to obtain a more general sufficient condition for weak ergodicity, consider the case where $\uptrans{}$ has a top class $\mathcal{R}$ and the process' initial state $x$ is in $\mathcal{R}$.
Since $\mathcal{R}$ is a maximal communication class, the process surely remains in $\mathcal{R}$ and hence, it is to be expected that the time average of $f$ will not be affected by the dynamics of the process outside $\mathcal{R}$.
Moreover, the communication class $\mathcal{R}$ is a strongly connected component, so one would expect that, due to Proposition~\ref{proposition: strongly connected then eigenvector}, the upper expected time average $\upprev{\mathrm{av},k}(f \vert x)$ converges to a constant that does not depend on the state $x \in \mathcal{R}$.
Our intuition is formalised by the following proposition.
Its proof, as well as those of the other statements in this section, \iftoggle{arxiv}{can be found in the appendix section}{are available in the appendix of \cite{arXivIPMU2020}}.

\begin{proposition}\label{proposition: if top class then eigenvector in top class}
For any maximal communication class $\mathcal{S}$ and any $x\in\mathcal{S}$, the upper expectation $\upprev{\mathrm{av},k}(f \vert x)$ is equal to $\upprev{\mathrm{av},k}(f \indica{\mathcal{S}} \vert x)$ and converges to a limit value. This limit value is furthermore the same for all $x \in \mathcal{S}$.
\end{proposition}

As a next step, we want to extend the domain of convergence of $\upprev{\mathrm{av},k}(f \vert x)$ to all states $x \in \statespace{}$.
To do so, we will impose the additional property of being top class absorbing (TCA), which, as explained in Section~\ref{Sect: accessibility and topical maps}, demands that there is a strictly positive (lower) probability to reach the top class $\mathcal{R}$ in a finite time period.
Once in $\mathcal{R}$, the process can never escape $\mathcal{R}$ though. One would therefore expect that as time progresses---as more of these finite time periods go by---this lower probability increases, implying that the process will eventually be in $\mathcal{R}$ with practical certainty.
Furthermore, if the process transitions from $x \in \mathcal{R}^c$ to a state $y \in \mathcal{R}$, then Proposition~\ref{proposition: if top class then eigenvector in top class} guarantees that $\upprev{\mathrm{av},k}(f \vert y)$ converges to a limit and that this limit value does not depend on the state $y$.
Finally, since the average is taken over a growing time interval, the initial finite number of time steps that it took for the process to transition from $x$ to $y$ will not influence the time average of $f$ in the limit. This leads us to suspect that $\upprev{\mathrm{av},k}(f \vert x)$ converges to the same limit as $\upprev{\mathrm{av},k}(f \vert y)$. Since this argument applies to any $x\in\mathcal{R}^c$, we are led to believe that $\uptrans{}$ is weakly ergodic. The following result confirms this.

\begin{proposition}\label{proposition: if top class absorbing then eigenvector}
Any $\uptrans{}$ that satisfies (TCA) is weakly ergodic. 
\end{proposition}

Conversely, suppose that $\uptrans{}$ does not satisfy (TCA).
Then there are two possibilities: 
either there is no top class or there is a top class but it is not absorbing.
If there is no top class, then it can be easily deduced that there are at least two maximal communication classes $\mathcal{S}_1$ and $\mathcal{S}_2$.
As discusssed earlier, the process cannot escape the classes $\mathcal{S}_1$ and $\mathcal{S}_2$ once it has reached them.
So if it starts in one of these communication classes, the process' dynamics outside this class are irrelevant for the behaviour of the resulting time average.
In particular, if we let $f$ be the function that takes the constant value $c_1$ in $\mathcal{S}_1$ and $c_2$ in $\mathcal{S}_2$, with $c_1 \not= c_2$, then we would expect that $\upprev{\mathrm{av},k}(f \vert x) = c_1$ and $\upprev{\mathrm{av},k}(f \vert y) = c_2$ for all $k \in \natz{}$, any $x \in \mathcal{S}_1$ and any $y \in \mathcal{S}_2$.
In fact, this can easily be formalised by means of Proposition~\ref{proposition: if top class then eigenvector in top class}.
Hence, $\upprev{\mathrm{av},\infty}(f \vert x)=c_1\neq c_2=\upprev{\mathrm{av},\infty}(f \vert y)$, so the upper transition operator $\uptrans{}$ cannot be weakly ergodic.

\begin{proposition}\label{prop: top class if weak ergodicity}
Any weakly ergodic $\uptrans{}$ has a top class.
\end{proposition} 

Finally, suppose that there is a top class $\mathcal{R}$, but that it is not absorbing.
This implies that there is an $x \in \mathcal{R}^c$ and a compatible precise model such that the process is guaranteed to remain in $\mathcal{R}^c$ given that it started in $x$.
If we now let $f = \indica{\mathcal{R}^c}$, then conditional on the fact that $X_0 = x$, the expected time average of $f$ corresponding to this precise model is equal to $1$. Furthermore, since $f\leq1$, no other process can yield a higher expected time average. The upper expected time average $\upprev{\mathrm{av},k}(f \vert x)$ is therefore equal to $1$ for all $k \in \natz{}$.
However, using Proposition~\ref{proposition: if top class then eigenvector in top class}, we can also show that $\upprev{\mathrm{av},k}(f \vert y) = 0$ for any $y \in \mathcal{R}$ and all $k \in \natz{}$.
Hence, $\upprev{\mathrm{av},\infty}(f \vert x)=1\neq 0=\upprev{\mathrm{av},\infty}(f \vert y)$, which precludes $\uptrans{}$ from being weakly ergodic.

\begin{proposition}\label{prop: TCA if weak ergodicity}
Any weakly ergodic $\uptrans{}$ that has a top class satisfies (TCA).
\end{proposition} 
Together with Propositions~\ref{proposition: if top class absorbing then eigenvector} and~\ref{prop: top class if weak ergodicity}, this allows us to conclude that (TCA) is a necessary and sufficient condition for weak ergodicity. 



\begin{theorem}\label{theorem: weakly ergodic iff top class absorbing}
$\uptrans{}$ is weakly ergodic if and only if it is top class absorbing.
\end{theorem}

\section{Conclusion}

The most important conclusion of our study of upper and lower expected time averages is its final result: that being top class absorbing is necessary and sufficient for weak ergodicity; a property that guarantees upper and lower expected time averages to converge to a limit value that does not depend on the process' initial state.
In comparison with standard ergodicity, which guarantees the existence of a limit upper and lower expectation, weak ergodicity thus requires less stringent conditions to be satisfied.
We illustrated this difference in Example~\ref{example 1}, where we considered a(n imprecise) Markov chain that satisfies (TCA) but not (TCR).

Apart from the fact that their existence is guaranteed under weaker conditions, the inferences $\upprev{\mathrm{av},\infty}(f)$ are also able to provide us with more information about how time averages might behave, compared to limit expectations.
To see why, recall Example~\ref{example 2}, where the inference $\upprev{\mathrm{av},\infty}(\indica{b}) = 1/2$ significantly differed from $\upprev{\infty}(\indica{b}) = 1$.
Clearly, the former was more representative for the limit behaviour of the time average of $\indica{b}$. 
As a consequence of \cite[Lemma~57]{8535240}, a similar statement holds for general functions.
In particular, it implies that $\upprev{\mathrm{av},\infty}(f) \leq \upprev{\infty}(f)$ for any function $f \in \setofgambles{}(\statespace{})$.
Since both inferences are upper bounds, $\upprev{\mathrm{av},\infty}(f)$ is therefore at least as informative as $\upprev{\infty}(f)$.

In summary then, when it comes to characterising long-term time averages, there are two advantages that (limits of) upper and lower expected time averages have over conventional limit upper and lower expectations: they exist under weaker conditions and they are at least as (and sometimes much more) informative.

That said, there is also one important feature that limit upper and lower expectations have, but that is currently still lacking for upper and lower expected time averages: an (imprecise) point-wise ergodic theorem~\cite[Theorem~32]{DECOOMAN201618}. For the limit upper and lower expectations of an ergodic imprecise Markov chain, this result states that
\begin{align*}
\smash{\lowprev{\infty}(f)} 
\leq \liminf_{k \to +\infty} \overline{f}_k (X_{0:k})
\leq \limsup_{k \to +\infty} \overline{f}_k (X_{0:k})
\leq \smash{\upprev{\infty}(f)},
\end{align*} 
with lower probability one. In order for limit upper and lower expected time averages to be the undisputed quantities of interest when studying long-term time averages, a similar result would need to be obtained for weak ergodicity, where the role of $\upprev{\infty}(f)$ and $\lowprev{\infty}(f) \coloneqq - \upprev{\infty}(- f)$ is taken over by $\upprev{\mathrm{av},\infty}(f)$ and $\smash{\lowprev{\mathrm{av},\infty}(f)} \coloneqq \smash{- \upprev{\mathrm{av},\infty}(- f)}$, respectively.
If such a result would hold, it would provide us with (strictly almost sure) bounds on the limit values attained by time averages that are not only more informative as the current ones, but also guaranteed to exist under weaker conditions.
Whether such a result indeed holds is an open problem that we would like to address in our future work.

A second line of future research that we would like to pursue consists in studying the convergence of $\upprev{\mathrm{av},k}(f \vert x)$ in general, without imposing that the limit value should not depend on $x$. We suspect that this kind of convergence will require no conditions at all.

\bibliographystyle{splncs04}
\bibliography{IPMU2020_TimeAverages}

\begin{thebibliography}{10}
\providecommand{\url}[1]{\texttt{#1}}
\providecommand{\urlprefix}{URL }
\providecommand{\doi}[1]{https://doi.org/#1}

\bibitem{Augustin:2014di}
Augustin, T., Coolen, F.P., de~Cooman, G., Troffaes, M.C.: {Introduction to
  Imprecise Probabilities}. John Wiley {\&} Sons, Chichester (2014)

\bibitem{DECOOMAN201618}
de~Cooman, G., {D}e Bock, J., Lopatatzidis, S.: Imprecise stochastic processes
  in discrete time: global models, imprecise {M}arkov chains, and ergodic
  theorems. International Journal of Approximate Reasoning  \textbf{76},
  18--46 (2016)

\bibitem{deCooman:2010gd}
de~Cooman, G., Hermans, F., Antonucci, A., Zaffalon, M.: {Epistemic irrelevance
  in credal nets: the case of imprecise Markov trees}. International Journal of
  Approximate Reasoning  \textbf{51}(9),  1029--1052 (2010)

\bibitem{deCooman:2009jz}
de~Cooman, G., Hermans, F., Quaeghebeur, E.: {Imprecise Markov chains and their
  limit behaviour}. Probability in the Engineering and Informational Sciences
  \textbf{23}(4),  597--635 (2009)

\bibitem{troffaes2014}
de~Cooman, G., Troffaes, M.C.: Lower Previsions. Wiley, Chichester (2014)

\bibitem{10.2307/3845053}
Gaubert, S., Gunawardena, J.: The {P}erron-{F}robenius theorem for homogeneous,
  monotone functions. Transactions of the American Mathematical Society
  \textbf{356}(12),  4931--4950 (2004)

\bibitem{GUNAWARDENA2003141}
Gunawardena, J.: From max-plus algebra to nonexpansive mappings: a nonlinear
  theory for discrete event systems. Theoretical Computer Science
  \textbf{293}(1),  141--167 (2003)

\bibitem{Hermans:2012ie}
Hermans, F., de~Cooman, G.: {Characterisation of ergodic upper transition
  operators}. International Journal of Approximate Reasoning  \textbf{53}(4),
  573--583 (2012)

\bibitem{HermansITIP}
Hermans, F., \v{S}kulj, D.: {Stochastic Processes}. In: Augustin, T., Coolen,
  F.P., de~Cooman, G., Troffaes, M.C. (eds.) Introduction to Imprecise
  Probabilities, pp. 258--278. John Wiley {\&} Sons, Chichester (2014)

\bibitem{kemeny1960finite}
Kemeny, J.G., Snell, J.L.: Finite {M}arkov chains. Undergraduate Text in
  Mathematics, Springer-Verlag, New York (1976)

\bibitem{extended8627473}
Krak, T., T'Joens, N., De~Bock, J.: Hitting times and probabilities for
  imprecise markov chains, extended version of \cite{8627473}, arXiv:1905.08781

\bibitem{8627473}
Krak, T., T'Joens, N., De~Bock, J.: Hitting times and probabilities for
  imprecise markov chains. In: International Symposium on Imprecise
  Probabilities : Theories and Applications, ISIPTA 2019, Proceedings.
  vol.~103, pp. 265--275. PMLR (2019)

\bibitem{8535240}
Lopatatzidis, S.: Robust Modelling and Optimisation in Stochastic Processes
  using Imprecise Probabilities, with an Application to Queueing Theory. Ph.D.
  thesis, Ghent University (2017)

\bibitem{10.1007/978-3-030-29765-7_38}
T'Joens, N., Krak, T., Bock, J.D., Cooman, G.d.: A recursive algorithm for
  computing inferences in imprecise markov chains. In: Proceedings of ECSQARU
  2019. pp. 455--465. Springer International Publishing (2019)

\bibitem{Walley:1991vk}
Walley, P.: Statistical Reasoning with Imprecise Probabilities. Chapman and
  Hall, London (1991)

\bibitem{Williams:2007eu}
Williams, P.M.: {Notes on conditional previsions}. International Journal of
  Approximate Reasoning  \textbf{44}(3),  366--383 (2007)

\end{thebibliography}

\iftoggle{arxiv}{

\appendix

\section{Proof of Proposition~\ref{proposition: if top class then eigenvector in top class}}

In the following, we will often use the fact that, since $\uptrans{}$ is coherent, the iterates of $\uptrans{}$ will also be coherent.
This can easily be derived using the coherence properties \ref{transcoherence: bounds}-\ref{transcoherence: mixed additivity} and an induction argument in $k$.
For an illustration of how to do so, we refer to \cite[Lemma~23]{extended8627473}.

\begin{lemma}\label{lemma: T^k is coherent}
If $\uptrans{}$ is a coherent upper transition operator then, for any $k \in \nats{}$, $\uptrans{}^k$ is coherent as well. 
\end{lemma}
The properties \ref{transcoherence: bounds}-\ref{transcoherence: mixed additivity} of coherent upper transition operators therefore also apply to $\uptrans{}^k$:
\begin{enumerate}[leftmargin=*,ref={\upshape{}C\arabic*$^\prime$},label={\upshape{}C\arabic*$^\prime$}.,itemsep=3pt, series=iteratedcoherence]
\item\label{iteratedcoherence: bounds} $\min h  \leq \uptrans{}^k h \leq \max h$ \hfill [boundedness];
\item\label{iteratedcoherence: subadditivity} $\uptrans{}^k (h+g) \leq \uptrans{}^k h + \uptrans{}^k g$ \hfill [sub-additivity];
\item\label{iteratedcoherence: homogeneity} $\uptrans{}^k (\lambda h) = \lambda \uptrans{}^k h$ \hfill [non-negative homogeneity];
\item\label{iteratedcoherence: constant addivity} $\uptrans{}^k (\mu + h) = \mu + \uptrans{}^k h$ \hfill [constant additivity];
\item\label{iteratedcoherence: monotonicity} if $h \leq g$ then $\uptrans{}^k h \leq \uptrans{}^k g$ \hfill [monotonicity];
\item\label{iteratedcoherence: mixed additivity} $\uptrans{}^k h - \uptrans{}^k g \leq \uptrans{}^k  (h-g)$ \hfill [mixed sub-additivity],
\end{enumerate}
for all $k \in \natz{}$, all $h,g \in \setofgambles{}(\statespace)$, all real $\mu$ and all real $\lambda \geq 0$.

Many of the results in this appendix will make use of the graph-theoretic concepts and notations that were defined in Section~\ref{Sect: accessibility and topical maps}.
Unless mentioned otherwise, we will always implicitly assume that they correspond to the graph $\mathscr{G}(\uptrans{})$ of $\uptrans{}$.
Note however that, due to Corollary~\ref{corollary: graphs are identical}, we could also equivalently consider the graphs $\mathscr{G}'(\uptrans{})$ or $\mathscr{G}'(\avuptrans{f}{})$.

\begin{lemma}\label{lemma: directed path}
\emph{\cite[Proposition~4]{Hermans:2012ie}}
For any two vertices $x$ and $y$, there is a directed path of length $k \in \nats{}$ from $x$ to $y$ if and only if $\uptrans{}^k \indica{y}(x) > 0$.
\end{lemma}

\begin{lemma}\label{lemma: not from S to Sc}
For any maximal communication class $\mathcal{S}$, we have that $\uptrans{}^k \indica{\mathcal{S}^c} (x) = 0$ for all $x \in \mathcal{S}$ and all $k \in \nats{}$. 
\end{lemma}
\begin{proof}
Consider any $x \in \mathcal{S}$.
Then, since $\mathcal{S}$ is maximal, we have that $x \not\to y$ for any $y \in \mathcal{S}^c$, which by Lemma~\ref{lemma: directed path} implies that $\uptrans{}^k \indica{y} (x) \leq 0$ for all $k \in \nats{}$.
Hence,
\begin{align*}
0 \leq \uptrans{}^k \indica{\mathcal{S}^c} (x) 
 = \Big[ \uptrans{}^k \Big( \sum\nolimits_{y \in \mathcal{S}^c} \indica{y} \Big) \Big] (x) \leq \sum\nolimits_{y \in \mathcal{S}^c} \uptrans{}^k \indica{y} (x) \leq 0 \text{ for all } k \in \nats{},
\end{align*}
where the first step uses \ref{iteratedcoherence: bounds} and the third uses \ref{iteratedcoherence: subadditivity}.
\qed 
\end{proof}
In the following, we will use $\supnorm{\cdot}$ to denote the supremum norm defined by $\supnorm{h} \coloneqq \max_{x \in \statespace{}} \vert h(x) \vert$ for all $h \in \setofgambles{}(\statespace{})$.

\begin{lemma}\label{lemma: time average in maximal class only depends on value of f in maximal class}
For any maximal communication class $\mathcal{S}$, we have that $\avuptrans{f}{}  h(x)  =  \avuptrans{f}{}  ( h \indica{\mathcal{S}} )(x)$ for all $h \in \setofgambles{}(\statespace{})$ and all $x \in \mathcal{S}$.
\end{lemma}
\begin{proof}
Fix any $h \in \setofgambles{}(\statespace{})$ and any $x \in \mathcal{S}$.
By sub-additivity [\ref{transcoherence: subadditivity}], we have that $\uptrans{}h(x) \leq \uptrans{}(h \indica{\mathcal{S}})(x) + \uptrans{}(h \indica{\mathcal{S}^c})(x)$.
Since $h \indica{\mathcal{S}^c} \leq \supnorm{h} \indica{\mathcal{S}^c}$, monotonicity [\ref{transcoherence: monotonicity}] therefore implies that  
\begin{align*}
\uptrans{}h(x) \leq \uptrans{}(h \indica{\mathcal{S}})(x) + \uptrans{}(\supnorm{h} \indica{\mathcal{S}^c})(x) 
&= \uptrans{}(h \indica{\mathcal{S}})(x) + \supnorm{h} \uptrans{}\indica{\mathcal{S}^c}(x) \\
&= \uptrans{}(h \indica{\mathcal{S}})(x),
\end{align*}
where the first equality follows from non-negative homogeneity [\ref{transcoherence: homogeneity}] and the second from Lemma~\ref{lemma: not from S to Sc}.
Hence, we obtain that $\avuptrans{f}{} h (x) \leq \avuptrans{f}{} (h \indica{\mathcal{S}}) (x)$.
To prove the converse inequality, observe that 
\begin{align*}
\uptrans{}h(x) \geq \uptrans{}(h \indica{\mathcal{S}})(x) - \uptrans{}( - h \indica{\mathcal{S}^c})(x)
&\geq \uptrans{}(h \indica{\mathcal{S}})(x) - \uptrans{}(\supnorm{h} \indica{\mathcal{S}^c})(x) \\
&= \uptrans{}(h \indica{\mathcal{S}})(x) - \supnorm{h} \uptrans{}\indica{\mathcal{S}^c}(x) \\
&= \uptrans{}(h \indica{\mathcal{S}})(x),
\end{align*}
where the first step follows from \ref{transcoherence: mixed additivity}, the second follows from $- h \indica{\mathcal{S}^c} \leq \supnorm{h} \indica{\mathcal{S}^c}$ and monotonicity [\ref{transcoherence: monotonicity}], the third follows from non-negative homogeneity [\ref{transcoherence: homogeneity}] and the last from Lemma~\ref{lemma: not from S to Sc}.
So, we have that $\uptrans{} h (x) \geq \uptrans{} (h \indica{\mathcal{S}}) (x)$ and therefore also that $\avuptrans{f}{} h (x) \geq \avuptrans{f}{} (h \indica{\mathcal{S}}) (x)$.
Hence, $\avuptrans{f}{} h (x) = \avuptrans{f}{} (h \indica{\mathcal{S}}) (x)$ for all $h \in \setofgambles{}(\statespace{})$ and all $x \in \mathcal{S}$.
\qed
\end{proof}


Consider any maximal communication class $\mathcal{S}$.
To prove Proposition~\ref{proposition: if top class then eigenvector in top class}, we will use the following notations that allow us to confine the dynamics of the process to the class $\mathcal{S}$.
For any $h \in \setofgambles{}(\statespace{})$, let $h \vert_{\mathcal{S}} \in \setofgambles{}(\mathcal{S})$ denote the restriction of $h$ to the domain $\mathcal{S}$.
Additionally, for any $h \in \setofgambles{}(\mathcal{S})$, we let $h^\uparrow \in \setofgambles{}(\statespace{})$ denote the zero-extension of $h$ into $\setofgambles{}(\statespace{})$, which takes the value $h(x)$ for $x \in \mathcal{S}$ and $0$ elsewhere.
Then note that $( h \vert_\mathcal{S} )^\uparrow = h \indica{\mathcal{S}}$ for any $h \in \setofgambles{}(\statespace{})$ and $( g^\uparrow )\vert_\mathcal{S} = g$ for any $g \in \setofgambles{}(\mathcal{S})$.
Let $\avuptrans{f,\mathcal{S}}{} \colon \setofgambles{}(\mathcal{S}) \to \setofgambles{}(\mathcal{S})$ be defined by $\avuptrans{f,\mathcal{S}}{} h \coloneqq (\avuptrans{f}{} h^\uparrow)\vert_{\mathcal{S}}$ for all $h \in \setofgambles{}(\mathcal{S})$.

\begin{lemma}\label{lemma: G^k is equal to G_R^k}
For any maximal communication class $\mathcal{S}$, we have that $(\avuptrans{f}{k} h)\vert_{\mathcal{S}} = \avuptrans{f,\mathcal{S}}{k} (h\vert_{\mathcal{S}})$ for all $h \in \setofgambles{}(\statespace{})$ and all $k \in \nats{}$.
\end{lemma}
\begin{proof}
We use an induction argument in $k \in \nats{}$.
That the statement holds for $k = 1$ follows immediately from Lemma~\ref{lemma: time average in maximal class only depends on value of f in maximal class}.
Indeed, for any $h \in \setofgambles{}(\statespace{})$, Lemma~\ref{lemma: time average in maximal class only depends on value of f in maximal class} says that $\avuptrans{f}{}h (x) = \avuptrans{f}{}(h \indica{\mathcal{S}}) (x)$ for all $x \in \mathcal{S}$ or, equivalently, that $(\avuptrans{f}{}h)\vert_{\mathcal{S}} = \big( \avuptrans{f}{}(h \indica{\mathcal{S}}) \big)\vert_{\mathcal{S}}$.
This implies, by the definition of $\avuptrans{f,\mathcal{S}}{}$ and the fact that $h \indica{\mathcal{S}} = (h \vert_{\mathcal{S}})^\uparrow$, that $(\avuptrans{f}{} h)\vert_{\mathcal{S}} = \avuptrans{f,\mathcal{S}}{} (h\vert_{\mathcal{S}})$ for any $h \in \setofgambles{}(\statespace{})$, which provides an induction base.

Now assume that the statement holds for all $i \in \{1,\cdots,k\}$, with $k \in \nats{}$.
Then, for any $h \in \setofgambles{}(\statespace{})$, we have that
\begin{align*}
(\avuptrans{f}{k+1} h)\vert_{\mathcal{S}} 
= \big( \avuptrans{f}{} (\avuptrans{f}{k} h) \big)\vert_{\mathcal{S}}
= \avuptrans{f,\mathcal{S}}{} \big((\avuptrans{f}{k} h)\vert_\mathcal{S} \big)
= \avuptrans{f,\mathcal{S}}{} \big( \avuptrans{f,\mathcal{S}}{k} (h\vert_\mathcal{S}) \big)
= \avuptrans{f,\mathcal{S}}{k+1} (h\vert_\mathcal{S}),
\end{align*}
where the second equality follows from the fact that the statement holds for $i=1$ and the third equality follows from the assumption that the statement holds for~$i=k$.
Combined with the induction base, this concludes the proof.
\qed
\end{proof}

\begin{lemma}\label{lemma: G_R is topical}
For any maximal communication class $\mathcal{S}$, the map $\avuptrans{f,\mathcal{S}}{}$ is topical.
\end{lemma}
\begin{proof}
To prove \ref{topical: constant addivity}, consider any $\mu \in \reals{}$ and any $h \in \setofgambles{}(\mathcal{S})$.
Since $\avuptrans{f}{}$ satisfies \ref{topical: constant addivity}, we have that $\avuptrans{f}{}(\mu + h^\uparrow) = \mu + \avuptrans{f}{}(h^\uparrow)$ and therefore, also that $\big( \avuptrans{f}{}(\mu + h^\uparrow) \big)\vert_\mathcal{S} = \mu + (\avuptrans{f}{} h^\uparrow )\vert_\mathcal{S} = \mu + \avuptrans{f,\mathcal{S}}{} h$.
Moreover, by Lemma~\ref{lemma: G^k is equal to G_R^k}, we have that $\big( \avuptrans{f}{}(\mu + h^\uparrow) \big)\vert_\mathcal{S} = \avuptrans{f,\mathcal{S}}{}\big((\mu + h^\uparrow)\vert_\mathcal{S} \big) = \avuptrans{f,\mathcal{S}}{}(\mu + h)$, implying that \ref{topical: constant addivity} holds.
Finally, that monotonicity [\ref{topical: monotonicity}] holds for $\avuptrans{f,\mathcal{S}}{}$ follows directly from its definition and the fact $\avuptrans{f}{}$ is monotone.
\qed
\end{proof}

\begin{lemma}\label{lemma: T_f,S has an eigenvector}
For any maximal communication class $\mathcal{S}$, the map $\avuptrans{f,\mathcal{S}}{}$ has an (additive) eigenvector.
\end{lemma}
\begin{proof}

Consider any two states $x$ and $y$ in $\mathcal{S}$.
Then, by definition, there is an edge from $x$ to $y$ in the graph $\mathscr{G}'(\avuptrans{f,\mathcal{S}}{})$ if $\lim_{\alpha \to +\infty} \avuptrans{f,\mathcal{S}}{}(\alpha \indica{y}) (x) = +\infty$.
Moreover, by the definition of $\avuptrans{f,\mathcal{S}}{}$, we have that $\avuptrans{f,\mathcal{S}}{}(\alpha \indica{y}) = \big(\avuptrans{f}{}(\alpha \indica{y})^\uparrow\big)\vert_\mathcal{S} = \big(\avuptrans{f}{}(\alpha \indica{y})\big)\vert_\mathcal{S}$ for all $\alpha \in \reals{}$, where we used $\indica{y}$ to denote the indicator of $y$ in both $\setofgambles{}(\mathcal{S})$ and $\setofgambles{}(\statespace{})$ depending on the domain of the considered map. 
Hence, there is an edge from $x$ to $y$ in the graph $\mathscr{G}'(\avuptrans{f,\mathcal{S}}{})$ if and only if $\lim_{\alpha \to +\infty} \avuptrans{f}{}(\alpha \indica{y}) (x) = +\infty$ or, equivalently, if and only if there is an edge from $x$ to $y$ in the graph $\mathscr{G}'(\avuptrans{f}{})$.
So $\mathscr{G}'(\avuptrans{f,\mathcal{S}}{})$ is identical to the restriction of the graph $\mathscr{G}'(\avuptrans{f}{})$ to the vertices in $\mathcal{S}$.
Now, $x$ and $y$ are two states in the maximal communication class $\mathcal{S}$ of $\mathscr{G}(\uptrans{})$, so we have that $x \to y$ in $\mathscr{G}(\uptrans{})$.
Moreover, the directed path from $x$ to $y$ remains within the maximal class $\mathcal{S}$, because $x \not\to z$ for any $x \in \mathcal{S}$ and any $z \in \mathcal{S}^c$.
Then, since $\mathscr{G}(\uptrans{})$ is identical to $\mathscr{G}'(\avuptrans{f}{})$ because of Corollary~\ref{corollary: graphs are identical}, and since $\mathscr{G}'(\avuptrans{f,\mathcal{S}}{})$ is the restriction of $\mathscr{G}'(\avuptrans{f}{})$ to $\mathcal{S}$, we find that $x \to y$ in $\mathscr{G}'(\avuptrans{f,\mathcal{S}}{})$.
Since this holds for any two vertices in $\mathscr{G}'(\avuptrans{f,\mathcal{S}}{})$, it follows that $\mathscr{G}'(\avuptrans{f,\mathcal{S}}{})$ is strongly connected.
Finally, $\avuptrans{f,\mathcal{S}}{}$ is also topical by Lemma~\ref{lemma: G_R is topical}, so Theorem~\ref{theorem: strongly connected then eigenvector} guarantees the existence of an (additive) eigenvector $h \in \setofgambles{}(\mathcal{S})$.
\qed
\end{proof}

\begin{lemma}\label{lemma: if top class then eigenvector in top class}
For any maximal communication class $\mathcal{S}$ and any $x \in \mathcal{S}$, the upper expected time average $\upprev{\mathrm{av},k}(f \vert x)$ converges to a constant that does not depend on the initial state $x$.
\end{lemma}
\begin{proof}
Consider any maximal communication class $\mathcal{S}$.
Lemma~\ref{lemma: T_f,S has an eigenvector} guarantees the existence of an eigenvector $h \in \setofgambles{}(\mathcal{S})$ of $\avuptrans{f,\mathcal{S}}{}$, so we have that 
\begin{align*}
\lim_{k \to +\infty} \avuptrans{f,\mathcal{S}}{k} (h) / k = \lim_{k \to +\infty} (h + k \mu) / k = \mu,
\end{align*}
where $\mu \in \reals{}$ is the eigenvalue corresponding to $h$.
Since $\avuptrans{f,\mathcal{S}}{}$ is topical due to Lemma~\ref{lemma: G_R is topical}, Lemma~\ref{lemma: cycle time exists independently of starting point} then also implies that $\smash{\lim_{k \to +\infty} \avuptrans{f,\mathcal{S}}{k} (0 \vert_{\mathcal{S}}) / k = \mu}$, with $0$ the zero vector in $\setofgambles{}(\statespace{})$.
Moreover, we have that
\begin{align*}
\upprev{\mathrm{av},k}(f \vert x) 
= \tfrac{1}{k+1} \big[ \, \avuptrans{f}{(k+1)} (0) \big](x) 
= \tfrac{1}{k+1} \big[ \, \avuptrans{f,\mathcal{S}}{(k+1)} (0\vert_{\mathcal{S}}) \big] (x)
\end{align*}
for all $x \in \mathcal{S}$ and all $k \in \natz{}$, where the first step follows from Equation~\eqref{Eq: recursive expression 2} and the second from Lemma~\ref{lemma: G^k is equal to G_R^k}.
This allows us to conclude that $\smash{\lim_{k \to +\infty} \upprev{\mathrm{av},k}(f \vert x) = \mu}$ for all $x \in \mathcal{S}$.
\qed
\end{proof}

\begin{lemma}\label{lemma: average does not depend on f outside S}
For any maximal communication class $\mathcal{S}$, we have that $\tilde{m}_{f,k} \indica{\mathcal{S}} = \tilde{m}_{g,k} \indica{\mathcal{S}}$ for any two $f,g \in \setofgambles{}(\statespace{})$ such that $g \indica{\mathcal{S}} = f \indica{\mathcal{S}}$ and all $k \in \natz{}$.
\end{lemma}
\begin{proof}
Let $\mathcal{S}$ be a maximal communication class.
Fix any two $f,g \in \setofgambles{}(\statespace{})$ such that $g \indica{\mathcal{S}} = f \indica{\mathcal{S}}$ and let $\avuptrans{f}{}(\cdot) \coloneqq f + \uptrans{}(\cdot)$ and $\avuptrans{g}{}(\cdot) \coloneqq g + \uptrans{}(\cdot)$ as before.
To prove the statement, we will use an induction argument in $k \in \natz{}$.
That the statement holds for $k = 0$ is trivial because $\tilde{m}_{f,0} = f$ and $\tilde{m}_{g,0} = g$.
Now suppose that the statement holds for $k = i-1$ with $i \in \nats{}$.
Then, by assumption, we have that $\tilde{m}_{f,i-1} \indica{\mathcal{S}} = \tilde{m}_{g,i-1} \indica{\mathcal{S}}$.
This allows us to write that, for any $x \in \mathcal{S}$, 
\begin{align}\label{Eq: lemma: average does not depend on f outside S}
\tilde{m}_{f,i}(x) 
= \avuptrans{f}{} \tilde{m}_{f,i-1} (x)
= \avuptrans{f}{} (\tilde{m}_{f,i-1} \indica{\mathcal{S}}) (x)
&= \avuptrans{f}{} (\tilde{m}_{g,i-1} \indica{\mathcal{S}}) (x) \nonumber \\
&= \avuptrans{f}{} \tilde{m}_{g,i-1}(x), 
\end{align}
where the second and last step follow from Lemma~\ref{lemma: time average in maximal class only depends on value of f in maximal class}.
Moreover, note that, for any $x \in \mathcal{S}$,
\begin{align*}
\avuptrans{f}{} \tilde{m}_{g,i-1}(x)
= (f +\uptrans{} \tilde{m}_{g,i-1})(x)
= f(x) +\uptrans{} \tilde{m}_{g,i-1}(x)
&= g(x) +\uptrans{} \tilde{m}_{g,i-1}(x) \\
&= \avuptrans{g}{} \tilde{m}_{g,i-1}(x) \\
&= \tilde{m}_{g,i}(x)
\end{align*}
where the second step follows from $f \indica{\mathcal{S}} = g \indica{\mathcal{S}}$.
Hence, recalling Equation~\eqref{Eq: lemma: average does not depend on f outside S}, we have that $\tilde{m}_{f,i}(x) = \tilde{m}_{g,i}(x)$ for all $x \in \mathcal{S}$, which implies that $\tilde{m}_{f,i} \indica{\mathcal{S}} = \tilde{m}_{g,i} \indica{\mathcal{S}}$ or, equivalently, that the statement holds for $k=i$.
\qed
\end{proof}

\begin{proofof}{Proposition~\ref{proposition: if top class then eigenvector in top class}.}
Consider any maximal communication class $\mathcal{S}$ and any $x \in \mathcal{S}$.
It follows from Lemma~\ref{lemma: average does not depend on f outside S} that for any $k \in \natz{}$, 
\begin{align*}
\upprev{\mathrm{av},k}(f \vert x) 
= \tfrac{1}{k+1} \tilde{m}_{f,k}(x) 
= \tfrac{1}{k+1} \tilde{m}_{f \indica{\mathcal{S}},k}(x)
= \upprev{\mathrm{av},k}(f \indica{\mathcal{S}} \vert x).
\end{align*}
Moreover, by Lemma~\ref{lemma: if top class then eigenvector in top class}, the upper expectation $\upprev{\mathrm{av},k}(f \vert x)$ converges to a constant that does not depend on the specific state $x \in \mathcal{S}$.
\end{proofof}

\section{Proof of Proposition~\ref{proposition: if top class absorbing then eigenvector}}

\begin{lemma}\label{lemma: Rc decreases}
For any $\uptrans{}$ with top class $\mathcal{R}$, the function $\uptrans{}^k \indica{\mathcal{R}^c}$ is non-increasing in $k \in \nats{}$. 
\end{lemma}
\begin{proof}
From Lemma~\ref{lemma: not from S to Sc} and the fact that a top class is always a maximal communication class, we infer that $\uptrans{} \indica{\mathcal{R}^c} (x) = 0$ for all $x \in \mathcal{R}$.
Since moreover $0 \leq \uptrans{} \indica{\mathcal{R}^c} \leq 1$ by \ref{transcoherence: bounds}, it follows that $\uptrans{} \indica{\mathcal{R}^c} \leq \indica{\mathcal{R}^c}$.
Then, using \ref{iteratedcoherence: monotonicity}, we deduce that 
$\uptrans{}^{k} \indica{\mathcal{R}^c} \leq \uptrans{}^{k-1} \indica{\mathcal{R}^c}$ for all $k \in \nats{}$.
\qed
\end{proof}

\begin{lemma}\label{lemma: limit of Rc converges to zero}
Consider any $\uptrans{}$ that satisfies (TCA) and let $\mathcal{R}$ be the corresponding top class.
Then we have that $\lim_{k \to +\infty} \uptrans{}^k \indica{\mathcal{R}^c} = 0$.
\end{lemma}
\begin{proof}
The statement holds if $\mathcal{R}^c = \emptyset$ because then $\indica{\mathcal{R}^c} = 0$, which by \ref{iteratedcoherence: bounds} implies that $\uptrans{}^k \indica{\mathcal{R}^c} = 0$ for all $k \in \natz{}$.
So assume that $\mathcal{R}^c$ is non-empty.
Then, since $\uptrans{}$ is top class absorbing, we know that for any $x \in \mathcal{R}^c$, there is an index $k_x \in \nats{}$ such that $\uptrans{}^{k_x} \indica{\mathcal{R}^c} (x) < 1$.
By Lemma~\ref{lemma: Rc decreases}, it follows that $\uptrans{}^{k} \indica{\mathcal{R}^c} (x) < 1$ for all $x \in \mathcal{R}^c$ and all $k \geq k_x$.
Hence, for $k \coloneqq \max_{x \in \mathcal{R}^c} k_x \in \nats{}$, we have that $\uptrans{}^{k} \indica{\mathcal{R}^c} (x) < 1$ for all $x \in \mathcal{R}^c$.
Now let $\alpha \coloneqq \max_{x \in \mathcal{R}^c} \uptrans{}^{k} \indica{\mathcal{R}^c} (x) < 1$.
The set $\mathcal{R}$ is a top class and therefore a maximal communication class, so it follows from Lemma~\ref{lemma: not from S to Sc} that $\uptrans{}^{k} \indica{\mathcal{R}^c}(x) = 0$ for all $x \in \mathcal{R}$.
Since $\uptrans{}^{k} \indica{\mathcal{R}^c}(x) \leq \alpha$ for all $x \in \mathcal{R}^c$, this implies that $\uptrans{}^{k} \indica{\mathcal{R}^c} \leq \alpha \indica{\mathcal{R}^c}$.
Using \ref{iteratedcoherence: monotonicity}, \ref{iteratedcoherence: homogeneity} and the non-negativity [\ref{iteratedcoherence: bounds}] of $\alpha$, it follows that $\uptrans{}^{2 k} \indica{\mathcal{R}^c} \leq \alpha^2 \indica{\mathcal{R}^c}$.
Repeating this argument leads us to conclude that $\uptrans{}^{\ell k} \indica{\mathcal{R}^c} \leq \alpha^\ell \indica{\mathcal{R}^c}$ for all $\ell \in \nats{}$.
Since $\alpha$ is a non-negative real such that $\alpha < 1$, 
this implies that $\lim_{\ell \to +\infty} \uptrans{}^{\ell k} \indica{\mathcal{R}^c} \leq 0$ and therefore by \ref{iteratedcoherence: bounds} that $\lim_{\ell \to +\infty} \uptrans{}^{\ell k} \indica{\mathcal{R}^c} = 0$.
Then it also follows that $\lim_{k \to +\infty} \uptrans{}^{k} \indica{\mathcal{R}^c} = 0$ because $\uptrans{}^{k} \indica{\mathcal{R}^c}$ is non-increasing [Lemma~\ref{lemma: Rc decreases}].
\qed
\end{proof}

\begin{lemma}\label{lemma: absorbing implies that T^k h only depends on h in R}
Consider any upper transition operator $\uptrans{}$ that satisfies (TCA) and let $\mathcal{R}$ be the corresponding top class.
Then, for any $\epsilon > 0$, there is a $k_1 \in \natz{}$ such that $\supnorm{\uptrans{}^k h -  \uptrans{}^k (h \indica{\mathcal{R}})} \leq \supnorm{h} \epsilon$ for all $k \geq k_1$ and all $h \in \setofgambles{}(\statespace{})$.
\end{lemma}
\begin{proof}
Fix any $\epsilon > 0$.
Because of Lemma~\ref{lemma: limit of Rc converges to zero}, we have that $\lim_{k \to +\infty} \uptrans{}^k \indica{\mathcal{R}^c} = 0$.
Since $\statespace{}$ is finite, this implies that there is an index $k_1 \in \natz{}$ such that $0 \leq \uptrans{}^k \indica{\mathcal{R}^c} \leq \epsilon$ for all $k \geq k_1$, where we also used the non-negativity [\ref{iteratedcoherence: bounds}] of $\uptrans{}^k \indica{\mathcal{R}^c}$.
Hence, multiplying by $\supnorm{h}$ for any $h \in \setofgambles{}(\statespace{})$ and using non-negative homogeneity [\ref{iteratedcoherence: homogeneity}], allows us to write that $0 \leq \uptrans{}^k (\supnorm{h} \indica{\mathcal{R}^c}) \leq \supnorm{h} \epsilon$ for all $k \geq k_1$ and all $h \in \setofgambles{}(\statespace{})$.
Moreover,
\begin{align*}
h \indica{\mathcal{R}} - \supnorm{h} \indica{\mathcal{R}^c} \leq h \leq h \indica{\mathcal{R}} + \supnorm{h} \indica{\mathcal{R}^c} \text{ for all } h \in \setofgambles{}(\statespace{}).
\end{align*}
Then, by subsequently applying \ref{iteratedcoherence: mixed additivity}, \ref{iteratedcoherence: monotonicity} and \ref{iteratedcoherence: subadditivity}, we find that 
\begin{align}\label{Eq: proposition: if top class absorbing then eigenvector 2}
\uptrans{}^k (h \indica{\mathcal{R}}) - \uptrans{}^k (\supnorm{h} \indica{\mathcal{R}^c}) 
&\leq \uptrans{}^k (h \indica{\mathcal{R}} - \supnorm{h} \indica{\mathcal{R}^c}) \nonumber \\
&\leq \uptrans{}^k h 
\leq \uptrans{}^k (h \indica{\mathcal{R}}) + \uptrans{}^k (\supnorm{h} \indica{\mathcal{R}^c}),
\end{align}
for all $h \in \setofgambles{}(\statespace{})$ and all $k \in \natz{}$.
Hence, recalling that $0 \leq \uptrans{}^k (\supnorm{h} \indica{\mathcal{R}^c}) \leq \supnorm{h} \epsilon$, we indeed find that $\supnorm{\uptrans{}^k h - \uptrans{}^k (h \indica{\mathcal{R}})} \leq \supnorm{h} \epsilon$ for all $h \in \setofgambles{}(\statespace{})$ and all $k \geq k_1$.
\qed
\end{proof}

\begin{lemma}\label{lemma: bounds average}
$\inf f \leq \upprev{\mathrm{av},k}(f \vert x) \leq \sup f$ for all $k \in \natz{}$ and all $x \in \statespace{}$.
\end{lemma}
\begin{proof}
It clearly suffices to prove that $(k+1) \inf f \leq \tilde{m}_{f,k} \leq (k+1) \sup f$ for all $k \in \natz{}$.
We do this by induction.
For $k=0$, the statement holds trivially because $\tilde{m}_{f,0} = f$ and therefore $\inf f \leq \tilde{m}_{f,0} \leq \sup f$.
Now suppose that the statement holds for $k = i-1$ with $i \in \nats{}$.
Then $i \inf f \leq \tilde{m}_{f,i-1} \leq i \sup f$.
It then follows from \ref{transcoherence: bounds} and \ref{transcoherence: monotonicity} that 
\begin{align*}
i \inf f = \uptrans{}(i \inf f) 
\leq \uptrans{} \tilde{m}_{f,i-1} 
\leq \uptrans{}(i \sup f)
= i \sup f.
\end{align*}
By adding $f$ to all the terms, we find that 
$(i+1) \inf f 
\leq f + \uptrans{} \tilde{m}_{f,i-1}  
\leq (i+1) \sup f$.
Since $f + \uptrans{} \tilde{m}_{f,i-1} = \avuptrans{f}{} \tilde{m}_{f,i-1}
= \tilde{m}_{f,i}$ by Equation~\eqref{Eq: recursive expression}, it follows that the statement holds for $k=i$ as well.
\qed
\end{proof}
For notational convenience, we will henceforth use $\overline{m}_{f,k} \coloneqq \tfrac{1}{k+1} \tilde{m}_{f,k}$ for any $k \in \natz{}$ to denote the function in $\setofgambles{}(\statespace{})$ that takes the value $\overline{m}_{f,k}(x) = \tfrac{1}{k+1} \tilde{m}_{f,k}(x) = \upprev{\mathrm{av},k}(f \vert x)$ in $x \in \statespace{}$. 

\begin{lemma}\label{lemma: T^kh - T_f^k h}
$\supnorm{\uptrans{}^k h - \avuptrans{f}{k} h } \leq k \supnorm{f}$ for all $h \in \setofgambles{}(\statespace{})$ and all $k \in \natz{}$.
\end{lemma}
\begin{proof}
We will only prove the statement for $k \in \nats{}$ since it clearly holds for $k = 0$.
To that end, it suffices to show that 
\begin{align}\label{Eq: lemma: average is equal to T^k average}
- k \supnorm{f} + \uptrans{}^k h \leq \avuptrans{f}{k} h \leq k \supnorm{f} + \uptrans{}^k h \text{ for all } h \in \setofgambles{}(\statespace{}) \text{ and all } k \in \nats{}.
\end{align}
We will use an induction argument in $k$.
It should be clear from the definition of $\avuptrans{f}{}$ that these inequalities hold for all $h \in \setofgambles{}(\statespace{})$ and $k = 1$.
Now suppose that they hold for all $h \in \setofgambles{}(\statespace{})$ and all $k \in \{1,\cdots,i\}$, with $i \in \nats{}$.
Then we have that 
\begin{align*}
\avuptrans{f}{i+1} h 
\leq \avuptrans{f}{}\big( i \supnorm{f} + \uptrans{}^i h \big)
= i \supnorm{f} + \avuptrans{f}{}\big(  \uptrans{}^i h \big)
\leq (i+1) \supnorm{f} +  \uptrans{}^{(i+1)} h,
\end{align*}
for all $h \in \setofgambles{}(\statespace{})$, where the first step follows from the induction hypothesis for $k=i$ and the monotonicity [\ref{topical: monotonicity}] of $\avuptrans{f}{}$, the second from the constant additivity [\ref{topical: constant addivity}] of $\smash{\avuptrans{f}{}}$, and the third from the induction hypothesis for $k=1$.
In an analogous way, we find that
\begin{align*}
\avuptrans{f}{i+1} h 
\geq \avuptrans{f}{}\big( - i \supnorm{f} + \uptrans{}^i h \big)
&= - i \supnorm{f} + \avuptrans{f}{}\big(  \uptrans{}^i h \big) \\
&\geq - (i+1) \supnorm{f} +  \uptrans{}^{(i+1)} h,
\end{align*}
for all $h \in \setofgambles{}(\statespace{})$, where the first step follows once more from the induction hypothesis for $k=i$ and the monotonicity [\ref{topical: monotonicity}] of $\avuptrans{f}{}$, the second from the constant additivity [\ref{topical: constant addivity}] of $\avuptrans{f}{}$, and the third from the induction hypothesis for $k=1$.
Both inequalities together establish that the statement holds for $k=i+1$, thereby concluding the induction step.
\qed
\end{proof}

\begin{lemma}\label{lemma: average is equal to T^k average}
$\lim_{k \to +\infty} \supnorm{\uptrans{}^{\ell} \overline{m}_{f,k} - \overline{m}_{f,k+\ell}} = 0$ \text{ for all } $\ell \in \natz{}$.
\end{lemma}
\begin{proof}
Fix any $\ell \in \natz{}$ and any $\epsilon>0$.
Let $k \in \natz{}$ be such that $k+1 \geq \ell \supnorm{f} / \epsilon$ and let $h \coloneqq \avuptrans{f}{k+1}(0) = (k+1) \overline{m}_{f,k}$.
Then
\begin{align}\label{Eq: lemma: average is equal to T^k average 2}
\epsilon 
\geq \tfrac{\ell}{k+1} \supnorm{f} 
\geq \tfrac{1}{k+1} \supnorm{\uptrans{}^\ell h - \avuptrans{f}{\ell} h }
&=  \supnorm{\uptrans{}^\ell \tfrac{1}{k+1} h - \tfrac{1}{k+1} \avuptrans{f}{\ell} h } \nonumber \\
&= \supnorm{\uptrans{}^\ell \overline{m}_{f,k} - \tfrac{1}{k+1} \avuptrans{f}{\ell} h },
\end{align}
where the second step follows from Lemma~\ref{lemma: T^kh - T_f^k h} and the third follows from the non-negative homogeneity [\ref{iteratedcoherence: homogeneity}] of $\uptrans{}^\ell$.
Moreover, we also have that
\begin{align*}
\supnorm{\tfrac{1}{k+1} \avuptrans{f}{\ell} h - \overline{m}_{f,k+\ell}}
= \supnorm{\tfrac{k+\ell+1}{k+1} \, \overline{m}_{f,k+\ell} - \overline{m}_{f,k+\ell}}
&= \tfrac{\ell}{k+1} \supnorm{\overline{m}_{f,k+\ell}} \\
&\leq \tfrac{\ell}{k+1} \supnorm{f} \leq \epsilon,
\end{align*}
where the second to last step follows from Lemma~\ref{lemma: bounds average}.
Combining this with Equation~\eqref{Eq: lemma: average is equal to T^k average 2} and using the triangle inequality, we get that
$\supnorm{\uptrans{}^\ell \overline{m}_{f,k} - \overline{m}_{f,k+\ell}}
\leq 2 \epsilon$.
Since this holds for any $\epsilon > 0$ and all $k + 1 \geq \ell \supnorm{f} / \epsilon $, we indeed have that $\lim_{k \to +\infty} \supnorm{\uptrans{}^\ell \overline{m}_{f,k} - \overline{m}_{f,k+\ell}} = 0$.
\qed
\end{proof}
The next proof uses the fact that any topical map $F \colon \setofgambles{}(\statespace{}) \to \setofgambles{}(\statespace{})$ is \emph{non-expansive} with respect to the supremum norm \cite[Proposition~1.1]{GUNAWARDENA2003141}, which means that:
\begin{align*}
\supnorm{Fh - Fg} \leq \supnorm{h-g} \text{ for all } h,g \in \setofgambles{}(\statespace{}).
\end{align*}

\begin{proofof}{Proposition~\ref{proposition: if top class absorbing then eigenvector}.}
Assume that $\uptrans{}$ satisfies (TCA) and let $\mathcal{R}$ be the corresponding top class.
We show that $\upprev{\mathrm{av},k} (f \vert x)$ converges to a constant that does not depend on $x \in \statespace{}$.
This is clearly the case if $f = 0$ because Lemma~\ref{lemma: bounds average} then implies that $\upprev{\mathrm{av},k} (f \vert x) = 0$ for all $x \in \statespace{}$ and $k \in \natz{}$.
So suppose that $f \not= 0$, fix any $\epsilon > 0$ and let $\epsilon_1 \coloneqq (\sfrac{1}{\supnorm{f}}) \epsilon$.
Choose $\ell_1$ such that Lemma~\ref{lemma: absorbing implies that T^k h only depends on h in R} holds with $\epsilon_1$.
Then, for any $\ell \geq \ell_1$, we have that
\begin{align*}
\supnorm{\uptrans{}^\ell (\overline{m}_{f,k} \indica{\mathcal{R}}) - \uptrans{}^\ell \overline{m}_{f,k}} 
\leq \epsilon_1 \supnorm{\overline{m}_{f,k}} 
\leq \epsilon_1 \supnorm{f}
= \epsilon \text{ for all } k \in \natz{},
\end{align*}
where the second inequality follows from Lemma~\ref{lemma: bounds average}.
Fix an $\ell \in \nats{}$ such that $\ell \geq \ell_1$.
Now, recall Lemma~\ref{lemma: average is equal to T^k average}, which guarantees that there is some $k_1 \in \natz{}$ such that $\supnorm{\uptrans{}^\ell \overline{m}_{f,k} - \overline{m}_{f,k+\ell}} \leq \epsilon$ for all $k \geq k_1$.
Combining this with the inequality above, we get that
\begin{align}\label{Eq: proposition: if top class absorbing then eigenvector 3}
\supnorm{\uptrans{}^\ell (\overline{m}_{f,k} \indica{\mathcal{R}}) - \overline{m}_{f,k+\ell}} \leq 2\epsilon \text{ for all } k \geq k_1.
\end{align}

The top class $\mathcal{R}$ is a maximal communication class, so Proposition~\ref{proposition: if top class then eigenvector in top class} guarantees that there is some $\mu \in \reals{}$ such that $\lim_{k \to +\infty} \overline{m}_{f,k}(x) = \mu$ for all $x \in \mathcal{R}$.
Then, since $\statespace{}$ is finite, there is some $k_2 \in \natz{}$ such that $\vert \overline{m}_{f,k}(x) - \mu \vert \leq \epsilon$ for all $k \geq k_2$ and all $x \in \mathcal{R}$.
Alternatively, we can also write that $\supnorm{\overline{m}_{f,k} \indica{\mathcal{R}} - \mu \indica{\mathcal{R}}} \leq \epsilon$ for all $k \geq k_2$.
The map $\uptrans{}$ is topical, implying that it is non-expansive and therefore that
\begin{align*}
\supnorm{\uptrans{}^\ell (\overline{m}_{f,k} \indica{\mathcal{R}}) - \uptrans{}^\ell (\mu \indica{\mathcal{R}})}
&\leq \supnorm{\uptrans{}^{\ell-1} (\overline{m}_{f,k} \indica{\mathcal{R}}) - \uptrans{}^{\ell-1} (\mu \indica{\mathcal{R}})} \\
&\leq \supnorm{\overline{m}_{f,k} \indica{\mathcal{R}} - \mu \indica{\mathcal{R}}} \leq \epsilon \,
\text{ for all } k \geq k_2.
\end{align*}
If we combine this with \eqref{Eq: proposition: if top class absorbing then eigenvector 3} and use the triangle inequality, we find that 
\begin{align}\label{Eq: proposition: if top class absorbing then eigenvector 4}
\supnorm{\uptrans{}^\ell (\mu \indica{\mathcal{R}}) - \overline{m}_{f,k+\ell}} \leq 3\epsilon \text{ for all } k \geq \max{(k_1,k_2)}.
\end{align}
Finally, we recall how $\ell \geq \ell_1$ was chosen and deduce that 
\begin{align*}
\supnorm{\uptrans{}^\ell (\mu \indica{\mathcal{R}}) - \mu}
= \supnorm{\uptrans{}^\ell (\mu \indica{\mathcal{R}}) - \uptrans{}^\ell \mu} 
\leq \epsilon_1 \supnorm{\mu} 
\leq \epsilon_1 \supnorm{f}
= \epsilon,
\end{align*}
where the first step follows from \ref{iteratedcoherence: bounds}.
To finish the proof, it suffices to combine this inequality with Equation~\eqref{Eq: proposition: if top class absorbing then eigenvector 4} using the triangle inequality.
This allows us to write that
$\supnorm{\mu - \overline{m}_{f,k+\ell}} \leq 4\epsilon$ for all $k \geq \max{(k_1,k_2)}$ or, equivalently, that  $\supnorm{\mu - \overline{m}_{f,k}} \leq 4\epsilon$ for all $k \geq \max{(k_1,k_2)} + \ell$.
Since this holds for any $\epsilon$ and since $\statespace{}$ is finite, we indeed have that $\lim_{k \to +\infty} \overline{m}_{f,k} = \mu$. 
Hence, the upper expectation $\upprev{\mathrm{av},k}(f \vert x) = \overline{m}_{f,k}(x)$ converges to the constant $\upprev{\mathrm{av},\infty}(f) \coloneqq \mu$ for all $x \in \statespace{}$.
\end{proofof}

\section{Proof of Theorem~\ref{theorem: weakly ergodic iff top class absorbing}}

Let us extend the domain of the relation $\rightarrow$ to all subsets of $\statespace{}$, by saying that $A \to B$, for any two $A,B \subseteq \statespace{}$, if $x \to y$ for all $x \in A$ and all $y \in B$.
If $A$ and $B$ are both communication classes, then we have that $A \to B$ if and only if $x \to y$ for at least one $x \in A$ and $y \in B$.

Furthermore, recall that $\leftrightarrow$ is an equivalence relation on $\statespace{}$.
Then it is well-known that the equivalence classes, called communication classes, form a partition $\mathscr{C}$ of $\statespace{}$.

\begin{lemma}\label{lemma: -> induces partial order}
The relation $\rightarrow$ induces a partial order on the partition $\mathscr{C}$ of all communication classes in $\statespace{}$.
\end{lemma}
\begin{proof}
That $\rightarrow$ is reflexive and transitive on $\mathscr{C}$ follows immediately from the reflexivity and transitivity of the relation $\to$ on the singletons.
To see that it is also antisymmetric, consider any two sets $A$ and $B$ in $\mathscr{C}$.
Then if $A \to B$ and $B \to A$, it should be clear that any two vertices in $A \cup B$ communicate and therefore that $A \cup B$ is a communication class.
Since $A$ and $B$ are two sets in the partition $\mathscr{C}$ of communication classes, $A \cup B$ can only be a communication class as well if $A=B$.
Hence, the relation $\to$ induces a partial order on the partition $\mathscr{C}$ of all communication classes.
\qed
\end{proof}

\begin{lemma}\label{lemma: if single communication class then top class}
Consider any upper transition operator $\uptrans{}$ that has a single unique maximal communication class $\mathcal{S}$.
Then $\mathcal{S}$ is the top class corresponding to $\uptrans{}$.
\end{lemma}
\begin{proof}
Consider any $x \in \mathcal{S}$ and any $y \in \statespace{}$.
We will show that $y \to x$, which by definition implies that $\mathcal{S}$ is the top class.
Let $\mathscr{C}$ be the partition of all communication classes in $\statespace{}$.
If $y$ is in the communication class $\mathcal{S}$, then $x$ and $y$ communicate and therefore $y \to x$.
So suppose that $y$ is not in $\mathcal{S}$.
Then, since $\mathscr{C}$ is a partition of $\statespace{}$, there is a unique $C_1 \in \mathscr{C}$ such that $y \in C_1$ and $C_1 \not= \mathcal{S}$.
As a consequence of our definition of $\to$ on the subsets of $\statespace{}$, it suffices to show that $C_1 \to \mathcal{S}$ in order to conclude that indeed $y \to x$.

Since $C_1 \not= \mathcal{S}$, the class $C_1$ cannot be maximal, so there is a second class $C_2 \in \mathscr{C}$ such that $C_1 \to C_2$ and $C_1 \not= C_2$.
Subsequently, if $C_2 \in \mathscr{C}$ also differs from $\mathcal{S}$, then there is a third class $C_3 \in \mathscr{C}$ such that $C_2 \to C_3$ and $C_2 \not= C_3$.
We also have that $C_1 \not= C_3$, because otherwise we would have that $C_1 \to C_2$ and $C_2 \to C_1$ and therefore by the antisymmetry of $\to$ that $C_1 = C_2$, which contradicts our assumptions.
Moreover, the transitivity of $\to$ implies that $C_1 \to C_3$.
Next, if $C_3 \in \mathscr{C}$ differs from $\mathcal{S}$, then there is a fourth class $C_4 \in \mathscr{C}$ such that $C_3 \to C_4$ and $C_3 \not= C_4$.
This class $C_4$ is also different from $C_1$ because otherwise we would have that $C_1 \to C_3$ and $C_3 \to C_1$ and therefore that $C_1 = C_3$, which once more contradicts our assumptions.
Similarly, since $C_2 \to C_3$ and $C_2 \not= C_3$, we deduce that $C_2 \not= C_4$.
Furthermore, $C_1 \to C_4$ by transitivity of $\to$.
It should be clear at this point that we can continue to repeat this argument, always obtaining a new communication class $C_n\in\mathscr{C}$ that differs from all previous ones and that is accessible from $C_1$.
Then, since $\statespace{}$---and therefore also $\mathscr{C}$---is finite, and since $\mathcal{S}\in\mathscr{C}$, we will eventually find that $C_n=\mathcal{S}$ is accessible from $C_1$, so $C_1\to C_n=\mathcal{S}$.
\qed
\end{proof}

\begin{corollary}\label{corollary: no top class}
Consider any upper transition operator $\uptrans{}$ that does not have a top class.
Then there are at least two maximal communication classes in $\statespace{}$.
\end{corollary}
\begin{proof}
Once more, let $\mathscr{C}$ be the partition of all communication classes in $\statespace{}$.
Since we know from Lemma~\ref{lemma: -> induces partial order} that $\to$ induces a partial order relation on $\mathscr{C}$ and because $\mathscr{C}$ is finite, there is at least one maximal communication class.
The case that there is exactly one is impossible, because by Lemma~\ref{lemma: if single communication class then top class} that would mean that there is a top class.
Hence, there are at least two maximal communication classes in $\statespace{}$.
\qed
\end{proof}

\begin{proofof}{Proposition~\ref{prop: top class if weak ergodicity}.}
Suppose that $\uptrans{}$ does not have a top class.
Then due to Corollary~\ref{corollary: no top class}, there are (at least) two maximal communication classes $\mathcal{S}_1$ and $\mathcal{S}_2$.
Consider any two $c_1, c_2 \in \reals{}$ such that $c_1 \not= c_2$, and let $f \coloneqq c_1 \indica{\mathcal{S}_1} + c_2 \indica{\mathcal{S}_2}$.
Since $f \indica{\mathcal{S}_1} = c_1 \indica{\mathcal{S}_1}$, Lemma~\ref{lemma: average does not depend on f outside S} implies that $\tilde{m}_{f,k} \indica{\mathcal{S}_1} = \tilde{m}_{c_1,k} \indica{\mathcal{S}_1}$ for all $k \in \nats{}$ or, equivalently, that $\upprev{\mathrm{av},k}(f \vert x) = \upprev{\mathrm{av},k}(c_1 \vert x)$ for all $x \in \mathcal{S}_1$ and all $k \in \nats{}$.
By Lemma~\ref{lemma: bounds average}, we know that for any $x \in \mathcal{S}_1$ and any $k \in \nats{}$, $\upprev{\mathrm{av},k}(f \vert x) = \upprev{\mathrm{av},k}(c_1 \vert x) = c_1$.
Hence, $\lim_{k \to +\infty} \upprev{\mathrm{av},k}(f \vert x) = c_1$ for all $x \in \mathcal{S}_1$.
In a completely analogous way, we can deduce that $\lim_{k \to +\infty} \upprev{\mathrm{av},k}(f \vert x) = c_2$ for all $x \in \mathcal{S}_2$.
By assumption, $c_1 \not= c_2$, so we can conclude that the upper expectation $\upprev{\mathrm{av},k}(f \vert x)$ with $f = c_1 \indica{\mathcal{S}_1} + c_2 \indica{\mathcal{S}_2}$, does not converge to a constant that is equal for all $x \in \statespace{}$.
Hence, the upper transition operator $\uptrans{}$ is not weakly ergodic.
\end{proofof}

\begin{lemma}\label{lemma: not TCA then TI_A equals one}
Consider any $\uptrans{}$ that has a top class $\mathcal{R}$ but that does not satisfy (TCA).
Then there is a non-empty subset $A \subseteq \mathcal{R}^c$ such that $\indica{A} \leq \uptrans{}\indica{A}$.
\end{lemma}
\begin{proof}
If $\uptrans{}$ has a top class $\mathcal{R}$ that is not absorbing, then there is at least one $x \in \mathcal{R}^c$ such that $\uptrans{}^k \indica{\mathcal{R}^c}(x) = 1$ for all $k \in \nats{}$.
Let $A \subseteq \mathcal{R}^c$ be the set of all such states $x \in \mathcal{R}^c$.
If $A = \mathcal{R}^c$ then, since $\uptrans{} \indica{\mathcal{R}^c}(x) = 1$ for all $x \in A = \mathcal{R}^c$ and since $\uptrans{} \indica{\mathcal{R}^c}(x) = 0$ for all $x \in \mathcal{R}$ by Lemma~\ref{lemma: not from S to Sc}, we have that $\uptrans{}\indica{\mathcal{R}^c} = \indica{\mathcal{R}^c}$.
Hence, in that case, the statement holds. In the remainder of the proof, we can therefore assume that $A \subset \mathcal{R}^c$, implying that $\mathcal{R}^c \setminus A$ is non-empty.

Observe that by the definition of $A$ there is for any $x \in \mathcal{R}^c \setminus A$, an index $k_x \in \nats{}$ such that $\uptrans{}^{k_x} \indica{\mathcal{R}^c} (x) \not= 1$, and therefore, due to \ref{iteratedcoherence: bounds}, also that $\uptrans{}^{k_x} \indica{\mathcal{R}^c} (x) < 1$.
By Lemma~\ref{lemma: Rc decreases}, it follows that then also $\uptrans{}^{k} \indica{\mathcal{R}^c} (x) < 1$ for all $x \in \mathcal{R}^c \setminus A$ and all $k \geq k_x$.
Hence, for $k \coloneqq \smash{\max_{x \in \mathcal{R}^c \setminus A} k_x} \in \nats{}$, we have that $\uptrans{}^{k} \indica{\mathcal{R}^c} (x) < 1$ for all $x \in \mathcal{R}^c \setminus A$.
Let $\alpha \coloneqq \smash{\max_{x \in \mathcal{R}^c \setminus A} \uptrans{}^{k} \indica{\mathcal{R}^c} (x)} < 1$.
Since $\uptrans{}^{k} \indica{\mathcal{R}^c}(x) = 0$ for all $x \in \mathcal{R}$ due to Lemma~\ref{lemma: not from S to Sc} and $0 \leq \alpha$ due to \ref{iteratedcoherence: bounds}, we infer that $\uptrans{}^{k} \indica{\mathcal{R}^c}(x) \leq \alpha$ for all $x \in \mathcal{R} \cup (\mathcal{R}^c \setminus A) = A^c$ or, equivalently, that $\indica{A^c} \uptrans{}^{k} \indica{\mathcal{R}^c} \leq \alpha \indica{A^c}$. Hence,
\begin{equation*}
\uptrans{}^k \indica{\mathcal{R}^c}
=\indica{A} \uptrans{}^k \indica{\mathcal{R}^c}+\indica{A^c}\uptrans{}^k \indica{\mathcal{R}^c}
\leq\indica{A} \uptrans{}^k \indica{\mathcal{R}^c}+ \alpha \indica{A^c}
=\indica{A}+ \alpha \indica{A^c},
\end{equation*}
using the definition of $A$ for the last equality. It follows that
\begin{align*}
\uptrans{}^{k+1} \indica{\mathcal{R}^c} 
=
\uptrans{}\,\uptrans{}^{k} \indica{\mathcal{R}^c} 
\leq \uptrans{} \big( \indica{A} + \alpha \indica{A^c} \big) 
= \uptrans{} \big( \alpha + (1 - \alpha) \indica{A} \big) 
= \alpha + (1 - \alpha) \uptrans{} \indica{A}, 
\end{align*}
using monotonicity [\ref{transcoherence: monotonicity}] for the inequality and \ref{transcoherence: constant addivity} and \ref{transcoherence: homogeneity} for the last equality. Multiplying with $\indica{A}$ yields $\indica{A} \uptrans{}^{k+1} \indica{\mathcal{R}^c} \leq \alpha\indica{A} + (1 - \alpha) \indica{A} \uptrans{} \indica{A}$ and therefore, since the definition of $A$ implies that $\indica{A}=\indica{A}\uptrans{}^{k+1} \indica{\mathcal{R}^c}$, we find that $\indica{A} \leq \alpha\indica{A} + (1 - \alpha) \indica{A} \uptrans{} \indica{A}$, or equivalently, that $(1 - \alpha) \indica{A} \leq (1 - \alpha) \indica{A} \uptrans{} \indica{A}$.
Since $1-\alpha > 0$, it follows that $\indica{A} \leq \indica{A} \uptrans{} \indica{A}$, which implies that $\indica{A} \leq \uptrans{} \indica{A}$ because $\uptrans{} \indica{A}$ is non-negative [\ref{transcoherence: bounds}].
\qed
\end{proof}

\medskip 

\begin{proofof}{Proposition~\ref{prop: TCA if weak ergodicity}.}
Consider any $\uptrans{}$ that has a top class $\mathcal{R}$ and suppose that $\mathcal{R}$ is not absorbing.
Then Lemma~\ref{lemma: not TCA then TI_A equals one} guarantees that there is a non-empty subset $A \subseteq \mathcal{R}^c$ such that $\indica{A} \leq \uptrans{}\indica{A}$.
Now consider any $x \in A$ and any $y \in \mathcal{R}$.
We will show that $\lim_{k \to +\infty} \upprev{\mathrm{av},k}(\indica{A} \vert x) = 1$ and that $\lim_{k \to +\infty} \upprev{\mathrm{av},k}(\indica{A} \vert y) = 0$, implying that $\uptrans{}$ cannot be weakly ergodic.

To prove that $\lim_{k \to +\infty} \upprev{\mathrm{av},k}(\indica{A} \vert x) = 1$, we show by induction that $\tilde{m}_{\indica{A},k} \geq (k+1) \indica{A}$ for all $k \in \natz{}$.
By definition, we have that $\tilde{m}_{\indica{A},0} = \indica{A}$, which establishes our induction base.
To prove the induction step, assume that the inequality holds for $k = i-1$, with $i \in \nats{}$, so $\tilde{m}_{\indica{A},i-1} \geq i \indica{A}$.
Then according to the recursive expression \eqref{Eq: recursive expression},
\begin{align*}
\tilde{m}_{\indica{A},i} 
= \indica{A} + \uptrans{} \tilde{m}_{\indica{A},i-1}
\geq \indica{A} + \uptrans{} (i \indica{A})
\geq (i+1) \indica{A},
\end{align*}
where the second step follows from the induction hypothesis and the monotonicity [\ref{transcoherence: monotonicity}] of $\uptrans{}$, and the last from \ref{transcoherence: homogeneity} together with the fact that $\indica{A} \leq \uptrans{}\indica{A}$.
This implies that the inequality holds for $k=i$ as well, hence finalising our induction argument.
We conclude that $\upprev{\mathrm{av},k}(\indica{A} \vert x) = \tfrac{1}{k+1} \tilde{m}_{\indica{A},k}(x) \geq \indica{A}(x)$ for all $k \in \natz{}$.
Due to Lemma~\ref{lemma: bounds average} and since $x \in A$, this implies that $\smash{\upprev{\mathrm{av},k}(\indica{A} \vert x)} = 1$ for all $k \in \natz{}$.
Hence, $\smash{\lim_{k \to +\infty} \upprev{\mathrm{av},k}(\indica{A} \vert x)} = 1$.

It remains to prove that $\lim_{k \to +\infty} \upprev{\mathrm{av},k}(\indica{A} \vert y) = 0$.
Because $A \subseteq \mathcal{R}^c$, we have that $\indica{A} \indica{\mathcal{R}} = 0 = 0 \, \indica{\mathcal{R}}$, and since $\mathcal{R}$ is a maximal communication class, Lemma~\ref{lemma: average does not depend on f outside S} implies that $\tilde{m}_{\indica{A},k} \indica{\mathcal{R}} = \tilde{m}_{0,k} \indica{\mathcal{R}}$ for all $k \in \natz{}$.
Hence, for any $k \in \natz{}$, we have that $\upprev{\mathrm{av},k}(\indica{A} \vert y) = \upprev{\mathrm{av},k}(0 \vert y) = 0$, where the last equality follows from Lemma~\ref{lemma: bounds average}.
As a consequence, $\lim_{k \to +\infty} \upprev{\mathrm{av},k}(\indica{A} \vert y) = 0$.
\end{proofof}

\bigskip

\begin{proofof}{Theorem~\ref{theorem: weakly ergodic iff top class absorbing}.}
That (TCA) is a sufficient condition follows from Proposition~\ref{proposition: if top class absorbing then eigenvector}.
Necessity follows from Proposition~\ref{prop: top class if weak ergodicity} together with Proposition~\ref{prop: TCA if weak ergodicity}.
\end{proofof}

}{}

\end{document}